\documentclass[12pt,british,reqno]{article}

\usepackage{babel,amsmath,amsthm,amsopn,amssymb,paralist,booktabs,url,eucal}
\usepackage{fancyhdr,tocloft}
\usepackage[margin=40pt,font=small,justification=centerlast]{caption}

\usepackage{ifpdf}
\ifpdf
  \usepackage{microtype}
  \usepackage[bookmarks,bookmarksnumbered=true]{hyperref}
\fi


\theoremstyle{plain}
\newtheorem{theorem}{Theorem}[section]
\newtheorem{proposition}[theorem]{Proposition}
\newtheorem{lemma}[theorem]{Lemma}
\newtheorem{corollary}[theorem]{Corollary}

\theoremstyle{definition}

\theoremstyle{remark}

\newtheorem{remark}[theorem]{Remark}
\newtheorem*{acknowledgements}{Acknowledgements}

\numberwithin{equation}{section}
\numberwithin{table}{section}

\setdefaultenum{\upshape (i)}{}{}{}
\setdefaultitem{--}{}{}{}

\newenvironment{smallpars}{\footnotesize \setlength{\parindent}{0pt}
\setlength{\parskip}{10pt}}{\par}

\newcommand{\Lie}[1]{\textsl{#1}}
\newcommand{\lie}[1]{\mathfrak{#1}}
\DeclareMathOperator{\GL}{\Lie{GL}}
\DeclareMathOperator{\Ort}{\Lie{O}}

\DeclareMathOperator{\SO}{\Lie{SO}}
\DeclareMathOperator{\so}{\lie{so}}
\DeclareMathOperator{\SP}{\Lie{Sp}}

\DeclareMathOperator{\Un}{\Lie{U}}

\newcommand{\inp}[2]{\langle #1, #2 \rangle}

\newcommand{\norm}[1]{\lVert #1 \rVert}

\DeclareMathOperator{\re}{Re}

\DeclareMathOperator{\Ric}{Ric}
\newcommand{\Ricq}{\Ric^{\textup{q}}}
\newcommand{\Ricqa}{\Ricq_{\mathbf a}}
\DeclareMathOperator{\scal}{scal}
\newcommand{\scalq}{\scal^{\textup{q}}}

\newcommand{\talt}{\tilde{\mathbf a}}
\newcommand{\taltb}{\tilde{\mathbf b}}
\newcommand{\sym}[1]{{S^{#1}\mskip-2mu}}

\newcommand{\Nt}{\widetilde\nabla}

\newcommand{\FCur}{\mathcal R}

\newcommand{\QK}{{\mathcal{QK}}}

\DeclareMathOperator{\End}{End}
\newcommand{\hook}{\lrcorner}

\newcommand{\T}{\mbox{\footnotesize \checkmark}}

\DeclareFontFamily{U}{bigeuf}{}
\DeclareFontShape{U}{bigeuf}{m}{n}{<-6>s*[1.5]eufm5%
<6-8>s*[1.5]eufm7%
<8->s*[1.5]eufm10}{}
\DeclareSymbolFont{bigeufletters}{U}{bigeuf}{m}{n}
\DeclareMathSymbol{\sumcic}{\mathop}{bigeufletters}{`S}

\newcommand{\bdash}{-\hspace{0pt}} 
\newcommand{\eqbreak}{\\&\qquad}
\newcommand{\smalleqbreak}{\\&\quad}

\newcommand{\itref}[1]{\textup{(\ref{#1})}}

\let\oldbibliography\thebibliography
\renewcommand{\thebibliography}[1]{%
  \oldbibliography{#1}%
  \setlength{\itemsep}{0pt}%
  \setlength{\parsep}{0pt}%
}

\begin{document}
\title{\bfseries Curvature of Almost Quaternion\bdash Hermitian Manifolds}

\author{Francisco Mart\'\i n Cabrera and Andrew Swann}

\date{ }

\maketitle

\begin{smallpars}
  \textbf{Abstract.} We study the decomposition of the Riemannian
  curvature~$R$ tensor of an almost quaternion-Hermitian manifold under the
  action of its structure group $Sp(n)Sp(1)$.  Using the minimal
  connection, we show that most components are determined by the intrinsic
  torsion~$\xi$ and its covariant derivative~$\widetilde\nabla\xi$ and
  determine relations between the decompositions of $\xi\otimes\xi$,
  $\widetilde\nabla\xi$ and~$R$.  We pay particular attention to the
  behaviour of the Ricci curvature and the q-Ricci curvature.


\end{smallpars}

\begin{center}
  \begin{minipage}{0.8\linewidth}
    \begin{small}
      \tableofcontents
    \end{small}
  \end{minipage}
\end{center}

\section{Introduction}

An object of fundamental importance in Riemannian geometry is the curvature
tensor~$R$.  As a $(0,4)$-tensor, $R$~satisfies a number of algebraic
symmetry conditions, including the Bianchi identity.  The presence of an
additional geometric structure on the manifold gives rise to a
decomposition of the curvature in to components, each satisfying additional
symmetry relations.  Further information on the geometric structure may
then imply the vanishing of some of these components.  The most celebrated
examples of this come from the holonomy classification of
Berger~\cite{Berger:hol} where nearly all the possible non-trivial
reductions of the Riemannian holonomy of an irreducible structure give
solutions of the Einstein equations, see~\cite{Besse:Einstein}.

Tricerri \& Vanhecke \cite{Tricerri-Vanhecke:aH} were the first to make a
detailed study of the general decompositions for one particular class of
geometric structures, namely almost Hermitian structures, i.e., manifolds
with a metric and compatible almost complex structure.  The purpose of this
paper is to extend these techniques to almost quaternion\bdash Hermitian
manifolds.  These are manifolds with a Riemannian metric $g=\inp\cdot\cdot$
and local triples of almost complex structures $I$, $J$, $K$ satisfying the
quaternion identities.

For the almost Hermitian case, the geometry is determined by a
$\Un(n)$\bdash structure.  For almost quaternion\bdash Hermitian manifolds
the structure group is $\SP(n)\SP(1)$.  Both groups appear on Berger's list
of Riemannian holonomy groups and are of fundamental importance in the
study of non-linear supersymmetric $\sigma$-models in physics.  

The first step in the study of curvature on these manifolds is to decompose
the space~$\FCur$ of $(0,4)$-curvature tensors under the action of the
structure group.  It is helpful to do this in two steps.  The space $\FCur$
is invariant under the larger group $\GL(n,\mathbb H)\SP(1)$ that preserves
the space of compatible almost complex structures but not the metric.  We
thus first decompose~$\FCur$ under the action of~$\GL(n,\mathbb H)\SP(1)$,
see~\S\ref{sec:quaternionic}, and then determine the refinement of this
decomposition under the action of the smaller group $\SP(n)\SP(1)$,
see~\S\ref{sec:aqH}.  This mirrors the approach of Falcitelli et
al.~\cite{Falcitelli-FS:aH} in the almost Hermitian case, see
also~\cite{Cabrera-S:su-curvature}.  Note that although $\GL(n,\mathbb
H)\SP(1)$ is the structure group of an almost quaternionic structure, the
metric has been used to convert the curvature tensor from type $(1,3)$ to
type $(0,4)$ by lowering an index, and so the first step does not
correspond to decomposition results for curvature of almost quaternionic
manifolds.

Given a Riemannian $G$-structure, there is a distinguished compatible
connection~$\Nt$, the minimal connection, characterised by having the
smallest torsion pointwise.  The torsion of $\Nt$ is called the intrinsic
torsion of the $G$-structure is determined by the tensor $\xi=\Nt-\nabla$,
where $\nabla$ is the Levi-Civita connection of the metric.  The vanishing
of~$\xi$ is equivalent to the reduction of the holonomy to~$G$.  In
general, the intrinsic torsion $\xi$ splits up in to a number of components
under the action of~$G$.  Vanishing of certain components often correspond
to interesting geometric properties, and structures with specific torsion
are of increasing importance in theoretical physics.

For $G=\SP(n)\SP(1)$ the intrinsic torsion~$\xi$ splits in to six
components \cite{Swann:symplectiques}.  These contribute to the curvature
tensor~$R$ in two different ways; via components of $\xi\otimes\xi$ and via
components $\Nt\xi$.  This information determines directly all the
components of $R$ transverse to the space~$\QK$ of curvature tensors of
quaternionic K\"ahler manifolds, i.e., manifolds where the holonomy reduces
to~$\SP(n)\SP(1)$.  Since $\dim\FCur = \tfrac43n^2(16n^2-1)$ and $\dim\QK =
\tfrac16(4n^4+12n^3+11n^2+3n+6)$, this means that nearly all components
of~$R$ are determined by~$\xi$ and its derivative.

In~\S\ref{sec:curvature}, we compute the contribution of~$\xi$ to the
components of~$R$ and display the results in tables.  Particular attention
is paid to the contributions to the Ricci curvature~$\Ric$ and its
quaternionic partner the q-Ricci curvature~$\Ricq$.  For these two tensors,
it is only the scalar parts that remain undetermined and we show how even
these may found by invoking additional information from the curvature of
the bundle of compatible local almost complex structures.  The paper closes
with a number of consequences for particular types of almost
quaternion\bdash Hermitian manifolds.  For examples of such structures we
refer the interested reader to Cabrera \& Swann
\cite{Cabrera-S:aqH-torsion}, where it is also shown how the components
of~$\xi$ may be efficiently computed via the exterior algebra.

\begin{acknowledgements}
  This work is supported by a grant from the MEC (Spain), project
  MTM2004-2644.  Francisco Mart\'\i n Cabrera thanks IMADA at the
  University of Southern Denmark in Odense for kind hospitality whilst
  working on this project.
\end{acknowledgements}

\section{Preliminaries}
\label{sec:preliminaries}

Let $\mathcal V$ be an $m$-dimensional real vector space.  The space of
Riemannian curvature tensors $\FCur$ on $\mathcal V$ consists of those
tensors $R$ of type $(0,4)$ which satisfies the same symmetries as the
Riemannian curvature tensor of a Riemannian manifold.  This is summarised
by saying that $\FCur$ is the kernel of the mapping
\begin{equation}
  \label{wedgingtwoforms}
  \sym 2 ( \Lambda^2 \mathcal V^* ) \to \Lambda^4 \mathcal V^*,
\end{equation}
defined by wedging two-forms together.

When there is a positive definite inner product
$g(\cdot,\cdot)=\inp\cdot\cdot$, defined on $\mathcal V$, then $\mathcal V$
is a representation of the orthogonal group $\Ort(m)$ and we can consider
the map $\Ric \colon \FCur \to \sym 2 \mathcal V^*$, given by $\Ric(R)(x,y)
= R(x,e_i,y,e_i)$, where $\{e_1, \dots, e_m \}$ is an orthonormal basis for
vectors and we use the summation convention.  This notation and such a
convention will be used in the sequel.  The map $\Ric$ is
$\Ort(m)$-invariant and $\Ric(R)$~is called the \emph{Ricci tensor}
associated to~$R$.  The \emph{scalar curvature} $\scal(R)$ of~$R$ is the
trace of~$\Ric(R)$.

There is a natural extension of the inner product $\inp\cdot\cdot$ to the
space of $p$-forms $\Lambda^p \mathcal V^*$ defined by
\begin{equation*}
  \inp ab = \tfrac1{p!} a(e_{i_1}, \dots, e_{i_p})
  b(e_{i_1},\dots, e_{i_p}).
\end{equation*}

On the other hand, we will also consider $\mathcal V$ equipped with three
almost complex structures $I$, $J$ and $K$ satisfying the same identities
as the imaginary units of quaternion numbers, i.e., $I^2 = J^2= -1$ and $K=
IJ = - JI$.  In such a case $m =4n$, and $\mathcal V$ is a representation
of the subgroup $\GL(n,\mathbb H)\SP(1)$ of $\GL(4n, \mathbb R)$
characterised by the fact that it preserves the three-dimensional vector space
$\mathcal G$ of endomorphisms of $\mathcal V$ generated by $I$, $J$ and
$K$.  A triple $I'$, $J'$ and $K'$ is said to be an \emph{adapted basis}
for $\mathcal G$, if they generate $\mathcal G$, satisfy the same
identities as the unit imaginary quaternions.  One finds that for $A =
I',J',K'$ we have $A = a_A I + b_A J + c_A K$ with $a_A^2 + b_A^2 + c_A^2 =
1$.

Furthermore, we will also consider both situations simultaneously:
$\mathcal V$ equipped with an inner product $\inp\cdot\cdot$ and three
almost complex structures $I$, $J$ and $K$ satisfying the above mentioned
quaternionic identities and the compatibility condition with the inner
product, $\inp{Ax}{Ay} = \inp xy$, for $A=I,J,K$.  In this case, we have the
three K{\"a}hler forms given by $\omega_A (x,y) = \inp x{Ay}$, $A=I,J,K$, and
the four-form $\Omega$ defined by
\begin{equation}
  \label{fundfourform}
  \Omega = \sum_{A=I,J,K} \omega_A \wedge \omega_A.
\end{equation}
The $4n$-form $\Omega^n$ can be used to fix an orientation and $\mathcal V$
is a representation of the subgroup $\SP(n)\SP(1)$ of $\SO(4n)$ consisting
by those elements which preserve $\mathcal G$.  Alternatively,
$\SP(n)\SP(1)$ can be defined as consisting of those elements of $\Ort(4n)$
which preserve $\Omega$.

The following notation will be used in this paper.  If $b$ is a
$(0,s)$-tensor, we write
\begin{gather*}
  A_{(i)}b(X_1,\dots,X_i,\dots,X_s) = - b(X_1,\dots,AX_i,\dots,X_s),\\
  A b(X_1,\dots,X_s) = (-1)^sb(AX_1,\dots,AX_s),
\end{gather*}
for $A=I,J,K$.

There is an $\SP(n)\SP(1)$-invariant map $\Ricq \colon \FCur \to \otimes^2
\mathcal V^*$ given by
\begin{equation*}
  \Ricq (R) (x,y) = \sum_{A=I,J,K} R(x,e_i,Ay,Ae_i).
\end{equation*}
The tensor $\Ricq (R)$ will be called the \emph{q-Ricci tensor} and the
trace $\scalq(R)$ of $\Ricq(R)$ will be referred as the \emph{q-scalar
curvature} of~$R$.  If we write
\begin{equation*}
  \Ric^*_A = R(x,e_i,Ay,Ae_i),
\end{equation*}
it is not hard to prove that
\begin{equation*}
  \sum_{A=I,J,K} A(\Ric^*_B)_{\mathbf a} = - (\Ric^*_B)_{\mathbf a}, \qquad
  \inp{\Ric^*_B}{\omega_B} = 0, 
\end{equation*}
where $(\Ric^*_B)_{\mathbf a}$ is the skew-symmetric part of $\Ric^*_B$.
Moreover, for all cyclic permutations of $I,J,K$, we have
\begin{equation*}
  \inp{\Ric^*_I}{\omega_J} = - \inp{\Ric^*_K}{\omega_J}.
\end{equation*}
Hence the skew-symmetric part $\Ricqa$ of the tensor $\Ricq$ satisfies the
conditions:
\begin{enumerate}
\item $\sum_{A=I,J,K} A \Ricqa = - \Ricqa$, and
\item $\inp{\Ricqa}{\omega_A} = 0$, for $A=I,J,K$,
\end{enumerate}
which characterises the irreducible $\SP(n)\SP(1)$-module $\Lambda^2_0 E
\sym 2 H \subset \Lambda^2 \mathcal V^*$ of skew-symmetric two-forms.  Thus
$\Ricqa \in \Lambda^2_0 E \sym 2 H$.  In summary, we have
\begin{align*}
  \Ricq &\in \sym 2 \mathcal V^* + \Lambda^2_0 E \sym 2 H = \mathbb R g +
  \Lambda_0^2 E + \sym 2 E \sym 2 H + \Lambda^2_0 E \sym 2 H\quad\text{and}\\
  \Ric &\in \sym 2 \mathcal V^* = \mathbb R g + \Lambda_0^2 E + \sym 2 E
  \sym 2 H,
\end{align*}
where $\Lambda^2_0 E$ consists of trace-free symmetric two-forms $b$ such
that $Ab=b$, $A=I,J,K$, and $\sym 2 E \sym 2 H$ consists of those such that
$\sum_A Ab=-b$.

Now, we recall some facts about quaternionic structures in relation with
representation theory.  We will follow the $E$-$H$-formalism used in
\cite{Salamon:Invent,Swann:symplectiques,Swann:MathAnn} and we refer to
\cite{Broecker-tom-Dieck:Lie} for general information on representation
theory.  Thus, $E$ is the fundamental representation of $\GL(n, \mathbb H)$
on ${\mathbb C}^{2n} \cong {\mathbb H}^n$ via left multiplication by
quaternionic matrices, considered in $\GL(2n, {\mathbb C})$, and $H$ is the
representation of $\SP(1)$ on ${\mathbb C}^2 \cong {\mathbb H}$ given by $q
.  \zeta = \zeta \overline{q}$, for $q \in \SP(1)$ and $\zeta \in \mathbb
H$.

An irreducible representation of $\GL(n, \mathbb H)$ is determined by its
dominant weight $(\lambda_1,\dots, \lambda_n)$, where $\lambda_i$ are
integers with $\lambda_1 \geqslant \lambda_2 \geqslant \dots \geqslant
\lambda_n \geqslant 0$.  This representation will be denoted by
$U^{\lambda_1, \dots, \lambda_r}$, where $r$ is the largest integer such
that $\lambda_r > 0$.  Likewise, $U^{\ast\lambda_1, \dots, \lambda_r}$ will
denote the dual representation of $U^{\lambda_1, \dots, \lambda_r}$.
Familiar notation is used for some of these modules, when possible.  For
instance, $U^{k} = \sym k E$, the $k$th symmetric power of $E$, and $U^{1,
\dots, 1} = \Lambda^r E$, where there are $r$ ones in exponent.  Likewise,
$\sym k E^*$ and $\Lambda^r E^*$ will be the respective dual
representations.

On $E$, there is an invariant complex symplectic form $\omega_E$ and a
Hermitian inner product given by $\inp xy_{\mathbb C} = \omega_E (x,
\widetilde{y})$, for all $x,y \in E$, and being $\widetilde{y} = j y $ ($y
\mapsto \widetilde{y}$ is a quaternionic structure map on $E={\mathbb
C}^{2n}$ considered as left complex vector space).  The mapping $x \mapsto
x^\omega = \omega_E ( \cdot, x)$ gives us an identification of $E$ with its
dual $E^*$.  In using group representations, this identification works only
with groups which preserve $\omega_E$.  For instance, we can not use such
an identification for $\GL(n, \mathbb H)$-representations.  If $\{ e_1,
\dots, e_n, \widetilde{e}_1, \dots, \widetilde{e}_n \}$ is a complex
orthonormal basis for $E$, then $ \omega_E = e^\omega_i \wedge
\widetilde{e}^\omega_i = e^\omega_i \widetilde{e}^\omega_i -
\widetilde{e}^\omega_i e^\omega_i, $ where we have omitted tensor product
signs.  The group $\SP(n)$ coincides with the subgroup of $\Un(2n)$ which
preserves $\omega_E$.

The $\SP(1)$-module $H$ will be also considered as left complex vector
space.  Regarding $H$ as 4-dimensional real space with the Euclidean metric
$\inp\cdot\cdot$ such that $\{ 1, i, j, k \}$ is an orthonormal basis, the
complex symplectic form $\omega_H$ is given by $\omega_H = (1^\flat \wedge
j^\flat + k^\flat \wedge i^\flat) + i (1^\flat \wedge k^\flat + i^\flat
\wedge j^\flat)$, where $h^\flat$ is the real one-form given by $q \mapsto
\inp hq$.  We also have the identification, $q \mapsto q^\omega = \omega_H
(\cdot, q )$, of $H$ with its dual $H^*$ as complex space.  On $H$, we have
a quaternionic structure map given by $h = z_1 + z_2 j \mapsto
\widetilde{h} = jq = - \overline{z}_2 + \overline{z}_1 j$, where $z_1,z_2
\in {\mathbb C}$ and $\overline{z}_1, \overline{z}_2 $ are their
conjugates.  On $H^*$, the structure map given by $\widetilde{h^\omega} = -
\widetilde{h}^\omega$, this is based in the identities
\begin{equation*}
  \widetilde{h^\omega}(q) = \overline{h^\omega(\widetilde{q})} =
  \overline{ \omega_H ( \widetilde{q}, h )} = - \omega_H ( q,
  \widetilde{h}) = - \widetilde{h}^\omega(q).
\end{equation*}
{}From now on, we will take $h \in H$ such that $\inp hh = 1$.  Thus $\{ h,
\widetilde{h} \}$ is a basis of the complex vector space 
$H$ and $\omega_H = h^\omega \wedge \widetilde{h}^\omega$.  Finally, we
point out that the irreducible representations of $\SP(1)$ are the
symmetric powers $\sym k H \cong {\mathbb C}^{k+1}$.

An irreducible representation of $\SP(n)$ is also determined by its
dominant weight $(\lambda_1,\dots,$ $ \lambda_n)$, where $\lambda_i$ are
integers with $\lambda_1 \geqslant \lambda_2 \geqslant \dots \geqslant
\lambda_n \geqslant 0$.  This representation will be denoted by
$V^{\lambda_1, \dots, \lambda_r}$, where $r$ is the largest integer such
that $\lambda_r > 0$.  Likewise, familiar notation is also used for some of
these modules.  For instance, $V^{k} = \sym k E$, and $V^{1, \dots, 1} =
\Lambda_0^r E$, where there are $r$ ones in exponent and $\Lambda_0^r E$ is
the $\SP(n)$-invariant complement to $\omega_E \Lambda^{r-2} E$ in
$\Lambda^r E$.  Also $K$ will denote the module $V^{21}$, which arises in
the decomposition $ E \otimes \Lambda_0^2 E \cong \Lambda_0^3 E + K + E$,
where $+$ denotes direct sum.

\begin{remark}
  Regarding complex and real representations: suppose $V$ is a complex
  $G$-module equipped with a real structure, where $G$ is a Lie group.
  Most of the time in this paper, $V$ will also denote the real $G$-module
  which is the $(+1)$-eigenspace of the structure map.  The context should
  tell us which space are referring to.  However, if there is risk of
  confusion or when we feel that a clearer exposition is needed, we will
  denote the second mentioned space by $[V]$.
\end{remark}

Returning to our real vector space $\mathcal V$ with the three almost
complex structures $I$, $J$ and $K$ satisfying the quaternionic identities,
we can consider $\mathcal V$ as complex vector space by saying $(\lambda +i
\mu)x= \lambda x + \mu Ix$, for all $\lambda+i \mu \in \mathbb C$ and $x
\in \mathcal V$.  Since $2n$ is the dimension of such a complex vector
space, we will also write $E$ when we are referring to $\mathcal V$ as
complex vector space.  The dual vector space $E^*$ of $E$ consists of
complex one-forms $a_{\mathbb C}= a + i I a$, for $a \in \mathcal V^*$, and
has $i a_{\mathbb C} = - (Ia)_{\mathbb C}$.  Because of the triple $I$, $J$
and~$K$ we can consider $E$ and $E^*$ as two complex $\GL(n, \mathbb
H)$-representations endowed with their respective quaternionic structure
maps defined by $x \mapsto \widetilde{x}=Jx$ and $ a_{\mathbb C} \mapsto
\widetilde{a_{\mathbb C}} = - (J a)_{\mathbb C}$.

The actions of $\GL(n,\mathbb H)\SP(1)$ on the real vector spaces $\mathcal
V$ and $\mathcal V^*$ gives rise to $\GL(n,\mathbb H)\SP(1)$-isomorphisms
which identify $\mathcal V \otimes_{\mathbb R} {\mathbb C} \cong E
\otimes_{\mathbb C} H$ and $\mathcal V^* \otimes_{\mathbb R} {\mathbb C}
\cong E^* \otimes_{\mathbb C} H^*$.  In fact, such isomorphisms are defined
respectively by $x \otimes_{\mathbb R} z \mapsto x \otimes_{\mathbb C} zh +
Jx \otimes_{\mathbb C} z \widetilde{h}$ and $a \otimes_{\mathbb R} z
\mapsto a_{\mathbb C} \otimes_{\mathbb C} zh^\omega + (Ja)_{\mathbb C}
\otimes_{\mathbb C} z \widetilde{h}^\omega$, where we have fixed $h \in H$
such that $\inp hh = 1$.

There are real structure maps defined on $E \otimes_{\mathbb C} H$ and $E^*
\otimes_{\mathbb C} H^*$ which are given by $x \otimes_{\mathbb C} q
\mapsto Jx \otimes_{\mathbb C} \widetilde{q}$ and $a_{\mathbb C}
\otimes_{\mathbb C} q^\omega \mapsto (Ja)_{\mathbb C} \otimes_{\mathbb C}
\widetilde{q}^\omega$, respectively.  Such structure maps correspond to $x
\otimes_{\mathbb R} z \mapsto x \otimes_{\mathbb R} \overline{z}$ and $a
\otimes_{\mathbb R} z \mapsto a \otimes_{\mathbb R} \overline{z}$
respectively defined on $\mathcal V \otimes_{\mathbb R} \mathbb C$ and
$\mathcal V^* \otimes_{\mathbb R} \mathbb C$.  Thus, for the corresponding
$(+1)$-eigenspaces, we have $[EH] \cong \mathcal V \otimes_{\mathbb R}
{\mathbb R} \cong \mathcal V$ and $[E^* H^*] \cong \mathcal V^*
\otimes_{\mathbb R} {\mathbb R} \cong \mathcal V^*$.

For considering elements of the tensorial algebra of $[E^* H^*] \cong
\mathcal V^*$, we need to compute the restrictions of $a_{\mathbb C}
h^\omega$, $a_{\mathbb C} \widetilde{h}^\omega$, $\widetilde{a_{\mathbb C}}
h^\omega$ and $\widetilde{a_{\mathbb C}} \widetilde{h}^\omega$ to $[EH]
\cong \mathcal V$.  Thus, for all $x \in \mathcal V$, we have $x \otimes h
+ Jx \otimes \widetilde{h}$ which is the corresponding element in $[EH]$
and obtain
\begin{equation} \label{ah}
  \begin{gathered}
    a_{\mathbb C} \otimes h^\omega ( x \otimes h + Jx \otimes
    \widetilde{h} ) =  (Ja - i K a)(x), \\
    \widetilde{a_{\mathbb C}} \otimes \widetilde{h}^\omega ( x \otimes h +
    Jx \otimes \widetilde{h} ) = (- Ja - i
    Ka) (x),  \\
    a_{\mathbb C} \otimes \widetilde{h}^\omega ( x \otimes h + Jx \otimes
    \widetilde{h} ) = (a + i Ia)
    (x),  \\
    \widetilde{a_{\mathbb C}} \otimes h^\omega ( x \otimes h + Jx \otimes
    \widetilde{h} ) = (a - i Ia) (x).
  \end{gathered}
\end{equation}

\section{Quaternionic decomposition of curvature}
\label{sec:quaternionic}

The space of Riemannian curvature tensors $\FCur$ is the kernel of the map
\eqref{wedgingtwoforms}, which is $\GL(n,\mathbb H)\SP(1)$-equivariant, so
$\FCur$ is also a $\GL(n,\mathbb H)\SP(1)$-module.  Our purpose here is to
show the $\GL(n,\mathbb H)\SP(1)$-decomposition of $\FCur$ into irreducible
components.

On the one hand, the space $\Lambda^2 \mathcal V^*$ of skew-symmetric
two-forms has the following decomposition into irreducible $\GL(n,\mathbb
H)\SP(1)$-modules,
\begin{equation}
  \label{twoforms:glnhdecomp}
  \Lambda^2 \mathcal V^* = \sym 2 E^* + \Lambda^2 E^* \sym 2 H,
\end{equation}
where the real module $\sym 2 E^*$ is characterised as consisting of those
$b \in \Lambda^2 \mathcal V^*$ such that $Ab=b$, for $A =I,J,K$, and the
skew-symmetric two-forms $b \in \Lambda^2 E^* \sym 2 H$ are such that
$\sum_{A=I,J,K} Ab = -b$.

Now, from \eqref{twoforms:glnhdecomp} it follows
\begin{equation*}
  \begin{split}
    \sym 2 ( \Lambda^2 \mathcal V^* ) &= \sym 2 ( \sym 2 E^* ) + \sym 2 (
    \Lambda^2 E^* \sym 2 H ) + \sym 2 E^* \Lambda^2 E^* \sym 2 H \\
    &= \sym 2 ( \sym 2 E^* ) + \sym 2 ( \Lambda^2 E^* ) ( \sym 4 H +
    \mathbb R ) \eqbreak + \Lambda^2 ( \Lambda^2 E^* ) \sym 2 H + \sym 2
    E^* \Lambda^2 E^* \sym 2 H,
  \end{split}
\end{equation*}
where we have taken $\sym 2 ( \sym 2 H ) \cong \sym 4 H + \mathbb R$ and
$\Lambda^2 ( \sym 2 H ) \cong \sym 2 H $ into account.  Since $\sym 2 (
\sym 2 E^* ) = \sym 4 E^* + U^{*22}$, $ \sym 2 ( \Lambda^2 E^* ) =
\Lambda^4 E^* + U^{*22}$, $\Lambda^2 ( \Lambda^2 E^* )= U^{*211}$ and $\sym
2 E \Lambda^2 E^*= U^{*211} + U^{*31}$, we obtain
\begin{equation*}
  \begin{split}
    \sym 2 ( \Lambda^2 \mathcal V^* ) &= \sym 4 E^* + 2 U^{*22} + \Lambda^4
    E^* + (U^{*31} + 2 U^{*211}) \sym 2 H \eqbreak + ( \Lambda^4 E^* +
    U^{*22}) \sym 4 H.
  \end{split}
\end{equation*}

On the other hand, for the skew symmetric four-forms on $\mathcal V$, we
obtain
\begin{equation*}
  \Lambda^4 \mathcal V^* = \Lambda^4 E^* \sym 4 H + U^{*211}
  \sym 2 H + U^{*22}.
\end{equation*}
Because there are non-vanishing values of the map \eqref{wedgingtwoforms}
on each one of these three summands, we conclude that
\begin{equation}
  \label{curvaturedec1}
  \FCur = \sym 4 E^* + U^{*22} + \Lambda^4 E^* + (U^{*31} + U^{*211}) \sym
  2 H + U^{*22} \sym 4 H. 
\end{equation}

In order to give explicit descriptions for these modules, we will consider
some $\GL(n,\mathbb H)\SP(1)$-endomorphisms on $\FCur$.  The first one $L$
is given by
\begin{equation}
  \label{firstglnh}
  L (R) = \sum_{\substack{1 \leqslant i < j \leqslant 4\\A=I,J,K}}
  A_{(i)}A_{(j)}R, 
\end{equation}
for all $R \in \FCur$.  Regarding $L$, we have the following results.

\begin{proposition}
  \label{glnhsplit1}
  The map $L$ is $\GL(n,\mathbb H)\SP(1)$-equivariant and
  \begin{asparaenum}
  \item\label{item:0H} $\sym 4 E^* + U^{*22} + \Lambda^4 E^*$ consists of
    $R \in \FCur$ such that $L(R)=6R$;
  \item\label{item:2H} $(U^{*31} + U^{*211}) \sym 2 H$ consists of $R \in
    \FCur$ such that $L(R)=2R$;
  \item\label{item:4H} $ U^{*22} \sym 4 H$ consists of $R \in \FCur$ such
    that $L(R)=-6R$.
  \end{asparaenum}
\end{proposition}

\begin{proof}
  If we use another adapted basis $I'$, $J'$ and $K'$ for $\mathcal G$ in
  equation~\eqref{firstglnh}, it is straightforward to check that we will
  obtain the same map $L$.  Hence $L$ is a $\GL(n,\mathbb H)\SP(1)$-map.

  For \itref{item:4H}, we first show that we have the following
  decomposition of $\sym 2 H \otimes \sym 2 H$ into $\SP(1)$-irreducible
  modules
  \begin{equation*}
    \sym 2 H \otimes \sym 2 H = \sym 2 (\sym 2 H) + \Lambda^2 (\sym2H) =
    \sym 4 H + \mathbb R \omega_H \otimes \omega_H + \sym2H, 
  \end{equation*}
  where we have taken $\sym 2 (\sym 2 H) \cong \sym 4 H + \mathbb R
  \omega_H \otimes \omega_H$ and $\Lambda^2 (\sym 2 H) \cong \sym 2 H$ into
  account.

  Next, we consider $(a_{\mathbb C} b_{\mathbb C} c_{\mathbb C} d_{\mathbb
  C}) \widetilde{h}^\omega \widetilde{h}^\omega \widetilde{h}^\omega
  \widetilde{h}^\omega \in (\otimes^4 E^*) \otimes \sym 4 H \subset
  \otimes^4 (E^* H)$, where we have omitted tensor product signs.  Let
  $\Phi_1 \in [ (\otimes^4 E) \sym 4 H] \subset \otimes^4 [E^* H]$ be the
  tensor defined by $\Phi_1 = \re ( (a_{\mathbb C}\widetilde{h}^\omega)
  (b_{\mathbb C} \widetilde{h}^\omega) (c_{\mathbb C} \widetilde{h}^\omega)
  (d_{\mathbb C} \widetilde{h}^\omega)_{|\mathcal V})$, where $\re$ means
  the real part.  Now, using equations~\eqref{ah}, we obtain
  \begin{equation*}
    \begin{split}
      \Phi_1 &= a b c d - a Ib Ic d - a Ib c Id - a b Ic Id \eqbreak - Ia
      Ib c d - Ia b Ic d - Ia b c Id + Ia Ib Ic Id.
    \end{split}
  \end{equation*}
  {}From this last expression it is straightforward to check that
  $L(\Phi_1)=-6\Phi_1$.  Since there are no conditions on $a$, $b$, $c$ and
  $d$, part~\itref{item:4H} follows.

  Part \itref{item:0H} follows by similar arguments considering
  $(a_{\mathbb C} b_{\mathbb C} c_{\mathbb C} d_{\mathbb C}) \omega_H
  \omega_H \in \otimes^4 E^* \subset \otimes^4 (E^* H)$.  Thus it is
  obtained $\Phi_2 \in [ \otimes^4 E ] \subset \otimes^4 [E^* H]$ defined
  by $\Phi_2 = \re ( (a_{\mathbb C} b_{\mathbb C} c_{\mathbb C} d_{\mathbb
  C}) \omega_H \omega_{H|\mathcal V})$, we recall that $\omega_H = h^\omega
  \widetilde{h}^\omega - \widetilde{h}^\omega h^\omega$.  After using
  equations~\eqref{ah}, one can check that $L(\Phi_2) = 6 \Phi_2$.

  Finally, for part~\itref{item:2H}, we recall that $(h^\omega h^\omega)
  \wedge ( \widetilde{h}^\omega \widetilde{h}^\omega) \in \Lambda^2 (\sym 2
  H) \cong \sym 2 H$ and consider
  \begin{equation*}
    (a_{\mathbb C} b_{\mathbb C} c_{\mathbb C} d_{\mathbb C}) ( h^\omega
    h^\omega \widetilde{h}^\omega \widetilde{h}^\omega -
    \widetilde{h}^\omega \widetilde{h}^\omega h^\omega h^\omega )\in
    (\otimes^4 E^*) \sym 2 H \subset 
    \otimes^4 (E^* H),
  \end{equation*}
  then, for $\Phi_3 = \re ( (a_{\mathbb C} b_{\mathbb C} c_{\mathbb C}
  d_{\mathbb C}) ( h^\omega h^\omega \widetilde{h}^\omega
  \widetilde{h}^\omega - \widetilde{h}^\omega \widetilde{h}^\omega h^\omega
  h^\omega )_{|\mathcal V})$, one can check that $L(\Phi_3) = 2 \Phi_3$.
\end{proof}

In order to go further with the descriptions of the $\GL(n,\mathbb
H)\SP(1)$\bdash submodules of the space of curvature tensors $\FCur$, we
will need to consider some $\GL(n,\mathbb H)\SP(1)$-maps from $\Lambda^2
\mathcal V^* \otimes \Lambda^2 \mathcal V^*$ to $\FCur$ which are defined
for $b,c \in \Lambda^2 \mathcal V^*$ by:
\begin{gather}
  \label{glnh1}
  \phi (b \otimes c) = 6 b \odot c - b \wedge c,\\
  \label{glnh3} \Phi (b \otimes c) = \sum_{A=I,J,K}
  \begin{aligned}[t]
    & \bigl( 6 (A_{(1)} + A_{(2)}) b \odot (A_{(1)}+ A_{(2)}) c \eqbreak -
    (A_{(1)}+A_{(2)}) b \wedge (A_{(1)}+A_{(2)}) c\bigr),
  \end{aligned}\\
  \label{glnh2}
  \varphi (b\otimes c)(x,y,z,u) = \sum_{A=I,J,K}
  \begin{aligned}[t]
    &\bigl( (A_{(1)} - A_{(2)}) b (x,z) (A_{(1)} - A_{(2)})c(y,u)
    \smalleqbreak - (A_{(1)} - A_{(2)}) b (x,u) (A_{(1)} - A_{(2)}) c(y,z)
    \smalleqbreak + (A_{(1)} - A_{(2)}) c (x,z) (A_{(1)} - A_{(2)}) b(y,u)
    \smalleqbreak - (A_{(1)} - A_{(2)}) c (x,u) (A_{(1)} - A_{(2)}) b(y,z)
    \bigr),
  \end{aligned}
\end{gather}
where we write $b \odot c = 1/2 (b \otimes c + c \otimes b)$ and $x, y, z,
u \in \mathcal V$.  Note that the maps $\phi$, $\varphi$ and $\Phi$
vanish on $\Lambda^2 ( \Lambda^2 \mathcal V^* )$, so we will consider
them as maps $\sym 2 ( \Lambda^2 \mathcal V^* )\to \FCur$.

Other $\GL(n,\mathbb H)\SP(1)$-maps that we will use are defined from $\sym
2 \mathcal V^* \otimes \sym 2 \mathcal V^*$ to $\FCur$.  These maps are
given for $b,c \in \sym 2 \mathcal V^*$ by:
\begin{gather}
  \allowdisplaybreaks
  \label{glnh4}
  \begin{split}
    \psi(b \otimes c) (x,y,z,u) & = b(x,z) c(y,u) - b(x,u) c(y,z) \eqbreak
    + c(x,z) b(y,u) - c(x,u) b(y,z),
  \end{split}\\
  \label{glnh5}
  \vartheta(b \otimes c) = \sum_{A=I,J,K}
  \begin{aligned}[t]
    &\bigl( 6 (A_{(1)} - A_{(2)}) b \odot (A_{(1)} - A_{(2)}) c \eqbreak -
    (A_{(1)} -A_{(2)}) b \wedge (A_{(1)}-A_{(2)}) c \bigr),
  \end{aligned}\\
  \label{glnh6}
  \Psi (b \otimes c) (x,y,z,u) = \sum_{A=I,J,K}
  \begin{aligned}[t]
    &\bigl((A_{(1)} + A_{(2)})b(x,z)(A_{(1)} + A_{(2)}) c(y,u) \eqbreak -
    (A_{(1)} + A_{(2)}) b(x,u) (A_{(1)} + A_{(2)}) c(y,z) \eqbreak +
    (A_{(1)} + A_{(2)}) c(x,z) (A_{(1)} + A_{(2)}) b(y,u) \eqbreak -
    (A_{(1)} + A_{(2)}) c(x,u) (A_{(1)} + A_{(2)}) b(y,z) \bigr),
  \end{aligned}
\end{gather}
for $x,y,z,u \in \mathcal V$.  Analogously, since $\psi$, $\vartheta$ and
$\Psi$ vanish on $\Lambda^2 ( \sym 2 \mathcal V^* )$, we will consider as
defined $\sym 2 ( \sym 2 \mathcal V^* ) \to \FCur$.

Likewise, a fundamental tool that we will use to describe the irreducible
$\GL(n,\mathbb H)\SP(1)$-modules of $\FCur$ is the $\GL(n,\mathbb
H)\SP(1)$-map $L_\sigma \colon \FCur \to \FCur$ which is defined by
\begin{equation}
  \label{secondglnh}
  L_\sigma (R) = \sum_{A=I,J,K}
  \begin{aligned}[t]
    ( &A_{(1)}A_{(2)} +A_{(2)}A_{(3)} \sigma + A_{(1)}A_{(3)} \sigma^2
    \smalleqbreak + A_{(3)}A_{(4)} +A_{(1)}A_{(4)} \sigma + A_{(2)}A_{(4)}
    \sigma^2 ) R,
  \end{aligned}
\end{equation}
where $\sigma = (123)$ is the permutation $1\mapsto 2\mapsto 3\mapsto 1$
and $\sigma R (x,y,z,u) = R(z,x,y,u)$.  As an illustration, $A_{(2)}A_{(3)}
\sigma R (x,y,z,u) = R(Az,x,Ay,u)$.

\begin{proposition}
  \label{glnhsplit2}
  For $L$ and $L_\sigma$ be as above, we have
  \begin{asparaenum}
  \item\label{item:6-12} $\sym 4 E^*$ consists of $R \in \FCur$ such that
    $L(R)=6R$ and $L_\sigma(R) = 12 R$;
  \item\label{item:6-0} $ U^{*22} $ consists of $R \in \FCur$ such that
    $L(R)=6R$ and $L_\sigma(R) = 0$;
  \item\label{item:6-m12} $\Lambda^4 E^*$ consists of $R \in \FCur$ such
    that $L(R)=6R$ and $L_\sigma(R) = -12 R$;
  \item\label{item:2-4} $U^{*31} \sym 2 H$ consists of $R \in \FCur$ such
    that $L(R)=2R$ and $L_\sigma(R) = 4 R$;
  \item\label{item:2-m4} $ U^{*211} \sym 2 H$ consists of $R \in \FCur$
    such that $L(R)=2R$ and $L_\sigma(R) = -4 R$;
  \item\label{item:m6-0} For all $ R \in U^{*22} \sym 4 H$, $L(R)=-6R$ and
    $L_\sigma(R) = 0$.
  \end{asparaenum}
\end{proposition}

\begin{proof}
  For \itref{item:6-12}, \itref{item:6-0} and~\itref{item:6-m12}, we
  consider $b_1,c_1 \in \sym 2 E^* \subset \Lambda^2 \mathcal V^* $.  Using
  equations \eqref{glnh1} and~\eqref{glnh2}, it is straightforward to check
  \begin{gather*}
    L_\sigma (4\phi(b_1\odot c_1) + \varphi(b_1\odot c_1)) = 12 ( 4
    \phi(b_1 \odot c_1) + \varphi (b_1\odot c_1)), \\
    L_\sigma (4\phi(b_1 \odot c_1) - \varphi(b_1 \odot c_1)) = 0.
  \end{gather*}
  Note that it is always possible to find $b_1$, $c_1$ such that
  $4\phi(b_1\odot c_1) + \varphi(b_1\odot c_1) \neq 0$ and $4\phi(b_1 \odot
  c_1) - \varphi(b_1 \odot c_1)\neq 0$.

  Since $b_1 \odot c_1 \in \sym 2 ( \sym 2 E^* ) = \sym 4 E^* + U^{*22}$,
  with Schur's Lemma in mind, then $4\phi(b_1\odot c_1) + \varphi(b_1\odot
  c_1) \in \sym 4 E^*$ and $4\phi(b_1\odot c_1) - \varphi(b_1\odot c_1) \in
  U^{*22}$, or $4\phi(b_1\odot c_1) + \varphi(b_1\odot c_1) \in U^{*22}$
  and $4\phi(b_1\odot c_1) - \varphi(b_1\odot c_1) \in \sym 4 E^*$.

  On the other hand, for $b_2,c_2 \in \Lambda^2 E \subset \sym 2 \mathcal
  V^*$, using equations \eqref{glnh4} and~\eqref{glnh5}, it is direct to
  check
  \begin{gather*}
    L_\sigma ( \vartheta(b_2 \odot c_2) + 12 \psi (b_2  \odot c_2) )=0, \\
    L_\sigma ( \vartheta(b_2 \odot c_2) - 12 \psi (b_2 \odot c_2) )= -12 (
    \vartheta(b_2 \odot c_2) - 12 \psi (b_2 \odot c_2) ).
  \end{gather*}
  Likewise, $\vartheta(b_2 \odot c_2) + 12 \psi (b_2 \odot c_2)$ and
  $\vartheta(b_2 \odot c_2) - 12 \psi (b_2 \odot c_2)$ are not always
  vanished.

  Since $b_2 \odot c_2 \in \sym 2 ( \Lambda^2 E^* ) = \Lambda^4 E^* +
  U^{*22}$, we deduce that $\vartheta(b_2 \odot c_2) - 12 \psi (b_2 \odot
  c_2) \in \Lambda^4 E^* $ and $\vartheta(b_2 \odot c_2) + 12 \psi (b_2
  \odot c_2) \in U^{*22}$, or $\vartheta(b_2 \odot c_2) - 12 \psi (b_2
  \odot c_2) \in U^{*22} $ and $\vartheta(b_2 \odot c_2) + 12 \psi (b_2
  \odot c_2) \in \Lambda^4 E^*$.

  Therefore,
  \begin{gather*}
    4\phi(b_1\odot c_1) + \varphi(b_1\odot c_1) \in \sym 4 E^*,\\
    4\phi(b_1\odot c_1) - \varphi(b_1\odot c_1),\ \vartheta(b_2 \odot
    c_2) + 12 \psi (b_2 \odot c_2) \in U^{*22},\\
    \vartheta(b_2 \odot c_2) - 12 \psi (b_2 \odot c_2) \in \Lambda^4 E^*.
  \end{gather*}
  Thus, taking Proposition \ref{glnhsplit1} into account,
  \itref{item:6-12}, \itref{item:6-0} and~\itref{item:6-m12} follow.

  For \itref{item:2-4} and~\itref{item:2-m4}, we consider $b_3 \in \sym 2
  E^* \subset \Lambda^2 \mathcal V^*$ and $c_3 \in \Lambda^2 E^* \sym 2 H$.
  Put $\alpha_1 = 4\phi(b_3\odot c_3) + \varphi(b_3 \odot c_3)$ and
  $\alpha_2 = 4 \phi(b_3\odot c_3) - 3 \varphi(b_3\odot c_3)$.  Then
  equations \eqref{glnh1} and~\eqref{glnh2} give
  \begin{equation*}
    L (\alpha_1) = 4\alpha_1\quad\text{and}\quad
    L (\alpha_2) = - 4\alpha_2,
  \end{equation*}
  so $\alpha_1$ and $\alpha_2$ belong to different irreducible summands of
  the space $\sym 2 E \otimes \Lambda^2 E^* \sym 2H = U^{*211} \sym 2H +
  U^{*31} \sym 2 H \subset \sym 2 ( \Lambda^2 \mathcal V^*)$ that contains
  $b_3\odot c_3$.  In Remark~\ref{sigmaus2h} below we will show that
  $\alpha_1\in U^{*211} \sym 2H$ and $\alpha_2\in U^{*31} \sym 2 H$,
  proving \itref{item:2-4} and~\itref{item:2-m4}.

  Finally, \itref{item:m6-0} will proved below, see
  Remark~\ref{lsigmaus4h}.
\end{proof}

\section{Almost quaternion-Hermitian decomposition of curvature}
\label{sec:aqH}

In~\S\ref{sec:quaternionic}, using the action of the Lie group
$\GL(n,\mathbb H)\SP(1)$ we obtained and described the decomposition of the
space of curvature tensors $\FCur$ given by equation~\eqref{curvaturedec1}.
In this section we will study the decompositions each one of these
submodules under the action of the subgroup $\SP(n)\SP(1)$ of
$\GL(n,\mathbb H)\SP(1)$.  As we have pointed out above, the main
difference between $\SP(n)$ and $\GL(n,\mathbb H)$ is that $\SP(n)$
preserves the complex symplectic form $\omega_E$.  Moreover, we have an
identification $E\cong E^*$ by $\omega_E$ and, consequently, all tensor
modules are identified with their corresponding duals.  Therefore, we will
write
\begin{equation*}
  \FCur = \sym 4 E + U^{22} + \Lambda^4 E + (U^{31} + U^{211})
  \sym 2 H + U^{22} \sym 4 H.
\end{equation*}

On the other hand, the presence of the metric $g=\inp\cdot\cdot$ allows to
work with the Ricci and q-Ricci curvature tensors.  Now we show the
relationships of these tensors with the maps $L$ and $L_\sigma$.

\begin{lemma}
  \label{ricci:LL}
  If $L$ and $L_\sigma$ are the $\SP(n)\SP(1)$-maps defined by equations
  \eqref{firstglnh} and~\eqref{secondglnh}, respectively, then
  \begin{gather*}
    \Ric(L(R))(X,Y) = 3 \Ric (X,Y) + \sum_{A=I,J,K} \Ric (AX,AY),\\
    \Ricq(L(R)) (X,Y) = 3 \Ricq (X,Y) + \sum_{A=I,J,K} \Ricq (AX,AY),\\
    \begin{split}
      \Ric(L_\sigma (R))(X,Y) &= 3 \Ricq (X,Y) + 3 \Ricq(Y,X) - 3 \Ric(X,Y)
      \eqbreak - \sum_{A=I,J,K} \Ric(AX,AY).
    \end{split}
  \end{gather*}
\end{lemma}

\begin{proof}
  It follows by straightforward computation.
\end{proof}

Now we will analyse the behaviour of the different $\GL(n,\mathbb
H)\SP(1)$\bdash submodules of $\FCur$ under the action of $\SP(n)\SP(1)$.
Since contractions by $\omega_E$ on the $\SP(n)$-module $\sym 4 E$ are all
zero, $\sym 4 E$ is also irreducible as an $\SP(n)$-module.  Therefore,
$\sym 4 E \cong \sym 4 E \otimes \mathbb C (\omega_H \otimes \omega_H)$ is
irreducible as an $\SP(n)\SP(1)$-module.

To study $U^{22}$, we consider $U^{22} \subset \sym 2 ( \Lambda^2 E )$ and
the map $\omega_{34} \colon U^{22} \to E \otimes E$ given by contraction
with $\omega_E$ on the $(3,4)$-indices.  One has that $\omega_{34} (U^{22})
= \Lambda^2 E$ and we write $V^{22}= \ker \omega_{34}$.  Therefore,
$U^{22}= V^{22} + \Lambda^2_0 E + \mathbb C \omega_E \otimes \omega_E$ is
the decomposition $U^{22}$ into $\SP(n)$-irreducible modules.  Then, for
$U^{22}$ as submodule of $\FCur$ and $n>1$, we have the following
decomposition into $\SP(n)\SP(1)$-irreducible summands,
\begin{equation*}
  U^{22}= V^{22} + (\Lambda^2_0 E)_a + \mathbb R_a,
\end{equation*}
where we have inserted the subscript~$a$ to distinguish these modules from
other copies of $\Lambda^2_0 E$ and $\mathbb R$ in~$\FCur$.  When $n>1$,
the three summands of the decomposition of $U^{22}$ are non-zero.  However,
if $n=1$, from $\sym 2 (\sym 2 E) = \sym 4 E + U^{22}$, we have $\dim
U^{22} = 1$.  Therefore, for $n=1$,
\begin{gather*}
  U^{22}= \mathbb R_a.
\end{gather*}

For providing detailed descriptions of these modules, in the next
Proposition we will need to consider
\begin{gather*}
  \pi_1(x,y,z,u) = \inp xz \inp yu - \inp xu \inp yz, \\
  \pi_2 = \sum_{A=I,J,K} ( 6 \omega_A \odot \omega_A -\omega_A \wedge
  \omega_A ).
\end{gather*}

\begin{proposition}
  \label{spn:curv1}
  Let $\vartheta$ and $\psi$ be the maps respectively defined by equations
  \eqref{glnh4} and~\eqref{glnh5}, then
  \begin{asparaenum}
  \item\label{item:6-0-0} $ V^{22} $ consists of $R \in \FCur$ such that
    $L(R)=6R$, $L_\sigma(R) = 0$ and $\Ric(R)=0$;
  \item\label{item:20Ea} $(\Lambda^2_0 E)_a$ consists of $R = \vartheta(b
    \otimes g) + 12 \psi(b \otimes g)$, where $b \in \Lambda_0^2 E \subset
    \sym 2 \mathcal V^*$ and $g=\inp\cdot\cdot$ is the metric.  Moreover,
    $\Ric(R)=48(n+1)b$;
  \item\label{item:Ra} $\mathbb R_a = \mathbb R(\vartheta(g \otimes g)+
    12\psi(g \otimes g))= \mathbb R (\pi_2 + 6 \pi_1)$.  Moreover,
    $\Ric(\pi_2 + 6 \pi_1)= 12(2n+1)g$.
  \item\label{item:22aa} If $R \in V^{22} + (\Lambda^2_0 E)_a + \mathbb
    R_a$, then $\Ric = \Ricq \in \Lambda^2_0 E + \mathbb R g$, i.e., for
    $A=I,J,K$, $A \,\Ric = \Ric$.
  \end{asparaenum}
\end{proposition}

\begin{proof}
  This follows from Propositions \ref{glnhsplit1} and~\ref{glnhsplit2}, the
  considerations in the proof of Proposition~\ref{glnhsplit2}, and the
  facts $\psi(g \otimes g)= 2 \pi_1$ and $\vartheta(g \otimes g) = 4
  \pi_2$.  Part~\itref{item:22aa} is a consequence of
  Proposition~\ref{glnhsplit2}\itref{item:6-0} and Lemma~\ref{ricci:LL}.
\end{proof}

Now let us consider the module $\Lambda^4 E$.  By contracting
with~$\omega_E$, we obtain the decomposition into irreducible
$\SP(n)$-modules given by $\Lambda^4 E = \Lambda^4_0 E + \omega_E \wedge
\Lambda_0^2 E + \mathbb C (\omega_E \wedge \omega_E)$.  Thus it follows
that the decomposition of $\Lambda^4 E \subset \FCur$ into irreducible
$\SP(n)\SP(1)$-modules is given by
\begin{equation}
  \label{l4e}
  \Lambda^4 E = \Lambda^4_0 E + (\Lambda_0^2 E)_b + \mathbb R_b.
\end{equation}
Note that:
\begin{asparaitem}
\item if $n>3$, then each one of the three summands is non-zero;
\item if $n=3$, then $\Lambda^4 E = (\Lambda_0^2 E)_b + \mathbb R_b$;
\item if $n=2$, then $\Lambda^4 E =\mathbb R_b$; and
\item if $n=1$, then $\Lambda^4 E =\{ 0 \}$.
\end{asparaitem}

Next we give more details relative to summands in the right side of
equation~\eqref{l4e}.

\begin{proposition}
  \label{spn:curv2}
  Let $\vartheta$ and $\psi$ be the maps defined respectively by equations
  \eqref{glnh4} and~\eqref{glnh5}, then
  \begin{asparaenum}
  \item\label{item:6-m12-0} $\Lambda^4_0 E $ consists of $R \in \FCur$ such
    that $L(R)=6R$, $L_\sigma(R) = -12 R$ and $\Ric(R)=0$;
  \item\label{item:20Eb} $(\Lambda^2_0 E)_b$ consists of $R=\vartheta(b
    \otimes g) - 12 \psi(b \otimes g)$, where $b \in \Lambda_0^2 E \subset
    \sym 2 \mathcal V^*$.  Moreover, $\Ric(R)=-48(n-2)b$;
  \item\label{item:Rb} $\mathbb R_b = \mathbb R(\vartheta(g \otimes g)-
    12\psi(g \otimes g))= \mathbb R (\pi_2-6 \pi_1)$.  Moreover,
    $\Ric(\pi_2-6 \pi_1)=-24(n-1)g$;
  \item\label{item:40Ebb} If $R \in \Lambda^4_0 E + (\Lambda^2_0 E)_b +
    \mathbb R_b$, then $ \Ric = - \Ricq \in \Lambda^2_0 E + \mathbb R g$,
    i.e., for $A=I,J,K$, $A \,\Ric = \Ric$.
  \end{asparaenum}
\end{proposition}

\begin{proof}
  This follows from Propositions \ref{glnhsplit1} and~\ref{glnhsplit2},
  considerations contained in the proof of Proposition~\ref{glnhsplit2} and
  Lemma~\ref{ricci:LL}.
\end{proof}

We have already pointed out that $\sym 2 E \Lambda^2 E = U^{31}+U^{211}$.
Moreover, one can check that $U^{31}= (\sym 3 E \otimes E ) \cap (\sym 2 E
\otimes \Lambda^2 E) $ and $U^{211} = ( E \otimes \Lambda^3 E) \cap (\sym 2
E \otimes \Lambda^2 E)$.  Therefore, contracting with~$\omega_E$, one
obtains the following decompositions into irreducible $\SP(n)$-summands
$U^{31} = V^{31} + \sym 2 E$ and $U^{211} = V^{211} + \sym 2 E +
\Lambda^2_0 E$.  Thus, for the modules $U^{31} \sym 2 H, U^{211} \sym 2 H
\subset \FCur$, we have the following decompositions into irreducible
$\SP(n)\SP(1)$-modules,
\begin{gather}
  \label{Utresuno}
  U^{31} \sym 2 H = V^{31} \sym 2 H + (\sym 2 E \sym 2 H)_a, \\
  \label{Udosunouno}
  U^{211} \sym 2 H = V^{211} \sym 2 H + (\sym 2 E \sym 2 H)_b + \Lambda^2_0
  E \sym 2 H.
\end{gather}
All of this happens for high dimensions.  However for low dimensions some
particular cases must be pointed out.
\begin{asparaitem}
\item For $U^{31}$.  If $n>1$, the two summands $V^{31}$ and $\sym 2 E$ are
  non-zero.  If $n=1$, then $\Lambda^2 E = \mathbb R \omega_E$ and $U^{31}
  = \sym 2E$.  Therefore, for $n=1$, we have
  \begin{equation*}
    U^{31} \sym 2 H = (\sym 2 E \sym 2 H)_a.
  \end{equation*}
\item For $U^{211}$.  If $n>2$, the three summands $V^{211}$, $\sym 2 E$
  and $\Lambda_0^2$ are non-zero.  If $n=2$, then $\Lambda^3 E = E \wedge
  \omega_E$ and $U^{211} = \sym 2 E + \Lambda^2_0 E$.  Therefore, for
  $n=2$, we have
  \begin{equation*}
    U^{211} \sym 2 H = (\sym 2 E \sym 2 H)_b + \Lambda^2_0 E \sym 2 H.  
  \end{equation*}
  If $n=1$, then $\Lambda^3 E = \{0\}$.  Therefore, $U^{211}=\{0\}$ and $
  U^{211} \sym 2 H=\{0\}$.
\end{asparaitem}

\begin{remark}
  \label{sigmaus2h}
  At this point we can complete the proof of parts \itref{item:2-4}
  and~\itref{item:2-m4} of Proposition~\ref{glnhsplit2}.  In fact, for $b
  \in \sym 2 E \subset \Lambda^2 \mathcal V^*$ and $c \in \Lambda^2 E \sym
  2 H\subset \Lambda^2 \mathcal V^*$, it is straightforward to check
  \begin{equation*}
    \Ricq(4 \phi(b \odot c) + \varphi (b\odot c))= 16 e_i \hook b \odot
    e_i \hook c - 8 \sum_{A=I,J,K} \inp{\omega_A}c A_{(1)}b,
  \end{equation*}
  where $\hook$ denotes contraction.  Thus, we have $\Ricq(4 \phi(b\odot c)
  + \varphi (b\odot c)) \in \sym 2 E \sym 2 H$.

  On the other hand,
  \begin{equation*}
    \begin{split}
      \Ricq(4 \phi(b \odot c) - 3 \varphi (b\odot c)) & = -40 e_i \hook b
      \otimes e_i \hook c + 24 e_i \hook c \otimes e_i \hook b \eqbreak - 8
      \sum_{A=I,J,K} \inp{\omega_A}c A_{(1)}b,
    \end{split}
  \end{equation*}
  which can have non-zero components in both $\sym 2 E \sym 2 H$ and
  $\Lambda_0^2 E \sym 2 H$.  Thus $\Ricq(4 \phi(b \odot c) - 3 \varphi (b
  \odot c)) \in \sym 2 E \sym 2 H + \Lambda_0^2 E \sym 2 H$.  All of this,
  taking equations \eqref{Utresuno} and~\eqref{Udosunouno} into account,
  implies $4 \phi(b \odot c) + 3 \varphi (b \odot c)) \in U^{31} \sym 2 H$
  and $4 \phi(b \odot c) - 3 \varphi (b \odot c) \in U^{211} \sym 2 H$.
\end{remark}

Now, we show more details for the summands of equations \eqref{Utresuno}
and~\eqref{Udosunouno}.

\begin{proposition}
  \label{spn:curv3}
  Let $\vartheta$ and $\psi$ be the maps defined respectively by equations
  \eqref{glnh4} and~\eqref{glnh5}, then
  \begin{asparaenum}
  \item\label{item:2-4-0} $V^{31} \sym 2 H $ consists of $R \in \FCur$ such
    that $L(R)=2R$, $L_\sigma(R) = 4 R$ and $\Ric(R)=0$;
  \item\label{item:22} $(\sym 2 E \sym 2 H)_a$ consists of $R=\vartheta(b
    \otimes g) + 4 \psi(b \otimes g)$, where $b \in \sym 2 E \sym 2 H
    \subset \sym 2 \mathcal V^*$.  Moreover, $\Ric(R) = 16 (n+1)b$;
  \item\label{item:2-m4-0} $V^{211} \sym 2 H $ consists of $R \in \FCur$
    such that $L(R)=2R$, $L_\sigma(R) = -4 R$ and $\Ric(R)=0$;
  \item\label{item:22b} $(\sym 2 E \sym 2 H)_b$ consists of $R=\vartheta(b
    \otimes g) - 12 \psi(b \otimes g)$, where $b \in \sym 2 E \sym 2 H
    \subset \sym 2 \mathcal V^*$.  Moreover, $\Ric(R) = -48
    (n-1)b$;
  \item\label{item:202} $\Lambda_0^2 E \sym 2 H$ consists of $R$ such that
    \begin{equation*}
      R = \sum_{A=I,J,K} \bigl(6 (A_{(1)} + A_{(2)} ) b \odot \omega_A
      - (A_{(1)} + A_{(2)} ) b \wedge \omega_A \bigr),
    \end{equation*}
    where $b \in \Lambda_0^2 E \sym 2 H \subset \Lambda^2 \mathcal V^*$.
    Moreover, $\Ricq (R) = -16 n b$;
  \item\label{item:3122a} if $R \in V^{31} \sym 2 H + (\sym 2 E \sym 2
    H)_a$, then $ \Ric = \Ricq \in \sym 2 E \sym 2 H$;
  \item\label{item:211b} if $R \in V^{211} \sym 2 H + (\sym 2 E \sym 2 H)_b
    + \Lambda_0^2 E \sym 2 H$ and $\Ricq_s$ denotes the symmetric part of
    $\Ricq$, then $ \Ric = - 3 \Ricq_s \in \sym 2 E \sym 2 H$.
  \end{asparaenum}
\end{proposition}

\begin{proof}
  For $b \in \sym 2 E \sym 2 H \subset \sym 2 \mathcal V^*$, it is not hard
  to check that
  \begin{gather*}
    L_\sigma (\vartheta(b \otimes g) + 4 \psi(b \otimes g)) = 4 (
    \vartheta(b \otimes g) +  4 \psi(b \otimes g) ),\\
    L_\sigma (\vartheta(b \otimes g) - 12 \psi(b \otimes g)) = - 4 (
    \vartheta(b \otimes g) - 12 \psi(b \otimes g)).
  \end{gather*}
  Now, all parts follow from Propositions \ref{glnhsplit1}
  and~\ref{glnhsplit2} and Lemma~\ref{ricci:LL}.
\end{proof}

Since we have already shown the $\SP(n)$-decomposition $U^{22}= V^{22} +
\Lambda^2_0 E + \mathbb C \omega_E \otimes \omega_E$, then, for $U^{22}
\sym 4 H \subset \FCur$ and $n>1$, we obtain
\begin{equation*}
  U^{22} \sym 4 H = V^{22} \sym 4 H + \Lambda^2_0 E \sym 4 H + \sym 4 H.
\end{equation*}
For $n=1$, as it was above pointed out, $U^{22}= \mathbb C$.  Therefore,
for $n=1$, we have
\begin{equation*}
  U^{22} \sym 4 H = \sym 4 H.
\end{equation*}

\begin{proposition}
  \label{spn:curv4}
  \begin{asparaenum}
  \item\label{item:m6} $V^{22} \sym 4 H $ consists of $R \in \FCur$ such
    that $L(R)=-6R$ and, for $A=I,J,K$, $\Ric^*_A (R)=0$;
  \item\label{item:204} $\Lambda_0^2 E \sym 4 H$ consists of $R$ such that
    \begin{equation}
      \label{les4h}
      R = \sum_{A=I,J,K} ( 6 b_A \odot \omega_A - b_A \wedge \omega_A ),
    \end{equation}
    where $b_I, b_J, b_K \in \Lambda_0^2 E \sym 2 H \subset \Lambda^2
    \mathcal V^*$ are such that $\sum_{A=I,J,K} A_{(2)} b_A = 0$;
  \item\label{item:4} $\sym 4 H $ consists of $R $ such that
    \begin{equation}
      \label{s4h}
      R = \sum_{A=I,J,K} ( 6 b_A \odot \omega_A - b_A \wedge \omega_A ),
    \end{equation}
    where $b_I, b_J, b_K \in \sym 2 H \subset \Lambda^2 \mathcal V^*$ are
    such that $\sum_{A=I,J,K} A_{(2)} b_A = 0$.
  \end{asparaenum}
\end{proposition}

\begin{proof}
  For~\itref{item:204}, if $R$ is given by equation~\eqref{les4h}, it is
  straightforward to check $L(R)=-6R$.  On the other hand, it is not hard
  to obtain, for $A=I,J,K$,
  \begin{gather*}
    \Ric^*_A = - 4 (n+1) A_{(2)} b_A \in \sym 2 E \sym 2 H + \Lambda_0^2 E
    \sym 2 H.
  \end{gather*}
  Hence $\Ricq=0$, but the local Ricci tensors $\Ric^*_A$ are not
  necessarily zero.

  For~\itref{item:4}, if $R$ is given by equation~\eqref{s4h}, where $ b_I
  = \lambda_{II} \omega_I + \lambda_{JI} \omega_J + \lambda_{KI} \omega_K$.
  It is also straightforward to check $L(R)=-6R$.  In this case, we have
  \begin{equation*}
    \Ric^*_I = 4 (n+1) ( \lambda_{II} g + \lambda_{KI} \omega_J -
    \lambda_{JI} \omega_K) \in \mathbb R g + \sym 2H,
  \end{equation*}
  and also $\Ricq=0$.  Since there are curvature tensors in the conditions
  of~\itref{item:m6}, part~\itref{item:m6} follows.
\end{proof}

\begin{remark}
  \label{lsigmaus4h}
  Now we will prove that $L_\sigma (R) =0$, for all $R \in U^{*22} \sym 4
  H$.  In fact, we consider $ R_1 = 6 \omega_I \odot \omega_J - \omega_I
  \wedge \omega_J \in \sym 4 H \subset U^{*22} \sym 4 H$.  It is direct to
  check $L_\sigma (R_1) =0$.  By Schur's Lemma, the assertion follows.
\end{remark}
  
\begin{remark}
  For a fixed adapted basis $I$, $J$, $K$ of $\mathcal G$, if $R \in
  \Lambda^2_0 E \sym 4 + \sym 4H$, then $R$ is determined by a unique
  triple $b_I$, $b_J$, $b_K$.  In \itref{item:4} of Proposition
  \ref{spn:curv4}, we can write the condition $\sum_{A=I,J,K}
  \inp{b_A}{\omega_A} = 0$ instead of $\sum_{A=I,J,K} A_{(2)} b_A = 0$, but
  in such a case more than one triple $b_I$, $b_J$, $b_K$ can determine the
  same element of $\sym 4H$.
\end{remark}

In summary, relative to the $\SP(n)\SP(1)$-decomposition of the space of
Riemannian curvature tensors $\FCur$, we have the following cases:
\begin{asparaitem}
\item if the dimension of $\mathcal{V}$ is strictly greater than $12$,
  $n>3$, then
  \begin{equation*}
    \begin{split}
      \FCur & = \sym 4 E + \bigl(\mathbb R_a + (\Lambda^2_0 E)_a +
      V^{22}\bigr) + \bigl( \mathbb R_b + (\Lambda_0^2 E)_b + \Lambda^4_0 E
      \bigr) \smalleqbreak + \bigl ((\sym 2 E \sym 2 H)_a + V^{31} \sym 2 H
      \bigr) + \bigl((\sym 2 E \sym 2 H)_b + \Lambda^2_0 E \sym 2 H +
      V^{211} \sym 2 H\bigr)\smalleqbreak + \bigl(\sym 4 H + \Lambda^2_0 E
      \sym 4 H + V^{22} \sym 4 H\bigr);
    \end{split}
  \end{equation*}
\item if the dimension of $\mathcal{V}$ is $12$, $n=3$, then
  \begin{equation*}
    \begin{split}
      \FCur & = \sym 4 E + \bigl(\mathbb R_a + (\Lambda^2_0 E)_a +
      V^{22}\bigr) + \bigl(\mathbb R_b + (\Lambda_0^2 E)_b\bigr)
      \smalleqbreak + \bigl((\sym 2 E \sym 2 H)_a + V^{31} \sym 2 H\bigr) +
      \bigl((\sym 2 E \sym 2 H)_b + \Lambda^2_0 E \sym 2 H + V^{211} \sym 2
      H\bigr) \smalleqbreak + \bigl( \sym 4 H + \Lambda^2_0 E \sym 4 H +
      V^{22} \sym 4 H\bigr);
    \end{split}
  \end{equation*}
\item if the dimension of $\mathcal{V}$ is $8$, $n=2$, then
  \begin{equation*}
    \begin{split}
      \FCur & = \sym 4 E + \bigl(\mathbb R_a + (\Lambda^2_0 E)_a +
      V^{22}\bigr) + \mathbb R_b \smalleqbreak + \bigl((\sym 2 E \sym 2
      H)_a + V^{31} \sym 2 H\bigr) + \bigl((\sym 2 E \sym 2 H)_b +
      \Lambda^2_0 E \sym 2 H \bigr) \smalleqbreak + \bigl( \sym 4 H +
      \Lambda^2_0 E \sym 4 H + V^{22} \sym 4 H \bigr);
    \end{split}
  \end{equation*}
\item and, if the dimension of $\mathcal{V}$ is $4$, $n=1$, then
  \begin{equation*}
    \FCur  =  \sym 4 E +  \mathbb R_a +  (\sym 2 E \sym 2 H)_a  +
    \sym 4 H.
  \end{equation*}
\end{asparaitem}

\section{Intrinsic torsion}
\label{sec:torsion}

Let $G$ be a subgroup of the linear group $\GL(m,\mathbb R)$.  A
manifold~$M$ is said to be equipped with a $G$-structure, if there is a
principal $G$-subbundle $P\to M$ of the principal frame bundle.  In this
situation, there always exist connections, called \emph{$G$-connections},
defined on the subbundle~$P$.  Moreover, if $(M^m, g=\inp\cdot\cdot)$ is an
orientable $m$-dimensional Riemannian manifold and $G$~is a closed and
connected subgroup of~$\SO(m)$, then there exists a unique metric
$G$-connection~$\Nt$ such that $\xi_x = \Nt_x - \nabla_x$ takes its values
in~$\lie g^\perp$, where $\lie g^\perp$ denotes the orthogonal complement
in~$\so(m)$ of the Lie algebra~$\lie g$ of~$G$ and $\nabla$~denotes the
Levi-Civita connection \cite{Salamon:holonomy,Cleyton-S:intrinsic}.  The
tensor~$\xi$ is the \emph{intrinsic torsion} of the $G$-structure and
$\Nt$~is called the \emph{minimal $G$-connection}.

A $4n$-dimensional manifold~$M$, $n > 1$, is said to be \emph{almost
quaternion\bdash Hermitian}, if $M$ is equipped with an
$\SP(n)\SP(1)$\bdash structure.  This is equivalent to the presence of a
Riemannian metric~$g=\inp\cdot\cdot$ and a rank-three subbundle~$\mathcal
G$ of the endomorphism bundle~$\End TM$, such that locally $\mathcal G$ has
an \emph{adapted basis} $I$, $J$, $K$ satisfying $I^2 = J^2 = -1$ and $K =
IJ = -JI$, and $\inp{AX}{AY} = \inp XY$, for all $X,Y \in T_x M$ and $A
=I,J,K$.  An almost quaternion-Hermitian manifold with a global adapted
basis is called an \emph{almost hyperHermitian} manifold.  The manifold is
then equipped with an $\SP(n)$-structure.

On each point~$p$ of these manifolds, the tangent space $T_p M$ can be
identified with the vector space~$\mathcal V$ of the previous section.
Thus there are three local K\"ahler-forms $\omega_A (X,Y) = \inp X{AY}$,
$A=I,J,K$.  From these one may define a global, non-degenerate
four-form~$\Omega$, the \emph{fundamental form}, by the local
formula~\eqref{fundfourform}.

In this section, we will recall some information about the intrinsic
torsion of almost quaternion-Hermitian manifolds.  More details may be
found in~\cite{Cabrera-S:aqH-torsion}, where it is also explained how to
explicitly compute the intrinsic torsion via the exterior algebra.

A connection~$\Nt$ is an $\SP(n)\SP(1)$-connection, if $\Nt \Omega =0$ or,
equivalently, if for any point of the manifold there exists a local adapted
basis $I$, $J$, $K$ of~$\mathcal G$ such that
\begin{gather*}
  (\Nt_X I)Y = \lambda_K(X) JY - \lambda_J(X) KY,\\
  (\Nt_X J)Y = \lambda_I(X) KY - \lambda_K(X) IY,\\
  (\Nt_X K)Y = \lambda_J(X) IY - \lambda_I(X) JY.
\end{gather*}
With respect to the Levi-Civita connection one then has formul\ae\ such as
\begin{equation}
  \label{torsion:sp(n)sp(1)}
  (\nabla_X I)Y = \lambda_K(X) JY - \lambda_J (X) KY - \xi_X IY + I\xi_X Y.
\end{equation}

\begin{proposition}[Cabrera \& Swann \cite{Cabrera-S:aqH-torsion}]
  The minimal $\SP(n)\SP(1)$\bdash connection is given by $\Nt = \nabla +
  \xi$, where $\nabla$ is the Levi-Civita connection and the intrinsic
  $\SP(n)\SP(1)$-torsion $\xi$ is defined by
  \begin{equation*}
    \xi_X Y = - \frac14 \sum_{A=I,J,K}
    A (\nabla_X A) Y + \frac12 \sum_{A=I,J,K} \lambda_A(X) A Y,
  \end{equation*}
  for all vectors $X,Y$, being the one-forms $\lambda_I$, $\lambda_J$ and
  $\lambda_K$ defined by cyclically permuting $I,J,K$ in the expression
  \begin{equation*}
    \lambda_I (X) = \frac1{2n} \inp{\nabla_X \omega_J}{\omega_K}.
  \end{equation*}
\end{proposition}

The next result describes the decomposition of the space $ T^*M \otimes
\Lambda_0^2 E \sym 2H$ of possible intrinsic torsion tensors into
irreducible $\SP(n)\SP(1)$-modules.

\begin{proposition}[Swann~\cite{Swann:symplectiques}]
  The intrinsic torsion $\xi$ of an almost quaternion\bdash Hermitian
  manifold $M$ of dimension at least $8$, has the property
  \begin{equation*}
    \xi \in T^* M \otimes \Lambda_0^2E \sym 2 H = ( \Lambda_0^3 E +
    K + E) ( \sym 3 H + H).
  \end{equation*}
\end{proposition}

\noindent If the dimension of $M$ is at least~$12$, all the modules of the
sum are non-zero.   For an eight-dimensional manifold $M$, we have
$\Lambda_0^3 E \sym 3 H = \Lambda_0^3 E H = \{ 0 \}$.  Therefore, for $\dim
M \geqslant 12 $ and $\dim M =8$, we have respectively $2^6 =64$ and $2^4=16$
classes of almost quaternion-Hermitian manifolds.  Explicit conditions
characterising these classes can be found in \cite{Cabrera:aqh}.

We use this Proposition to decompose $\xi$ as 
\begin{equation*}
  \xi = \xi_{33} + \xi_{K3} + \xi_{E3} + \xi_{3H} + \xi_{KH} +
  \xi_{EH},
\end{equation*}
where $\xi_{UF} \in U\otimes F$, for $U=\Lambda^3_0E, K, E$ and $F=\sym3H,
H$.  The components of the intrinsic torsion~$\xi$ have the following
specific symmetry properties and characterisations described
in~\cite{Cabrera-S:aqH-torsion}.
\begin{asparaenum}
  \setdefaultleftmargin{}{6em}{}{}{}{} 
\item $\xi_{33}$ is a tensor characterised by the conditions:
  \begin{enumerate}
  \item $\sum_{A=I,J,K} (\xi_{33})_A A = - \sum_{A=I,J,K} A (\xi_{33})_A =
    - \xi_{33}$,
  \item $\inp\cdot{(\xi_{33})_\cdot\cdot}$ is a skew-symmetric three-form.
  \end{enumerate}
\item $\xi_{K3}$ is a tensor characterised by the conditions:
  \begin{enumerate}
  \item $\sum_{A=I,J,K} (\xi_{K3})_A A = -\sum_{A=I,J,K} A (\xi_{K3})_A = -
    \xi_{K3}$,
  \item $\sumcic_{XYZ} \inp Y{(\xi_{K3})_X Z} =0$.
  \end{enumerate}
\item $\xi_{E3}$ is given by
  \begin{multline*}
    \inp Y{(\xi_{E3})_X Z} \\
    = \tfrac1n\sum_{A=I,J,K} \bigl( nA ( \theta^\xi_A - \theta^\xi ) \wedge
    \omega_A - (n-1) A ( \theta^\xi_A - \theta^\xi) \otimes
    \omega_A\bigr)(X,Y,Z),
  \end{multline*}
  where $\theta^\xi$ is the one-form defined by
  \begin{equation}
    \label{eq:global-theta}
    \tfrac6n (2n+1)(n-1) \theta^\xi (X) = - \inp{\xi_{e_i} e_i}X = -
    \sum_{A=I,J,K} \inp{A\xi_{e_i} A e_i}X,
  \end{equation}
  and $\theta^\xi_I$, $\theta^\xi_J$, $\theta^\xi_K$ are the local
  one-forms given by
  \begin{equation*}
    \tfrac2n (2n+1) (n-1) \theta^\xi_A (X) = - \inp{A \xi_{e_i} A e_i}X.
  \end{equation*}
  Note that $3\theta^\xi = \theta^\xi_I + \theta^\xi_J + \theta^\xi_K$.
\item $\xi_{3H}$ is a tensor characterised by the conditions:
  \begin{enumerate}
  \item $(\xi_{3H})_A A - A (\xi_{3H})_A -A \xi_{3H} A = \xi_{3H}$, for
    $A=I,J,K$,
  \item $\sumcic_{X,Y,Z} \inp Y{(\xi_{3H})_X Z} = 0$.
  \end{enumerate}
\item $\xi_{KH}$ is a tensor characterised by the conditions:
  \begin{enumerate}
  \item $ (\xi_{KH})_A A - A (\xi_{KH})_A - A \xi_{KH} A = \xi_{KH}$, for
    $A=I,J,K$;
  \item there exists a skew-symmetric three-form $\psi^{(K)}$ such that
    \begin{equation*}
      \inp Y{(\xi_{KH})_X Z} =\bigl(3\psi^{(K)} - \sum_{A=I,J,K} A_{(23)}
      \psi^{(K)} \bigr)(X,Y,Z);
    \end{equation*}
  \item $\sum_{i=1}^{4n} (\xi_{KH})_{e_i} e_i =0$.
  \end{enumerate}
\item $\xi_{EH}$ is given by
  \begin{equation*}
    \begin{split}
      \inp Y{(\xi_{EH})_X Z} &= 3 e_i \otimes e_i \wedge \theta^\xi(X,Y,Z)
      \eqbreak - \sum_{A=I,J,K} \bigl( e_i \otimes A e_i \wedge A
      \theta^\xi + \tfrac2{n} A \theta^\xi \otimes \omega_A \bigr) (X,Y,Z),
    \end{split}
  \end{equation*}
  where $\theta^\xi$ is the global one-form defined
  by~\eqref{eq:global-theta}.
\item The part $\xi_{\sym3H} = \xi_{33} + \xi_{K3} + \xi_{E3} $ of $\xi$ in
  $( \Lambda_0^2 E + K + E) \sym3H$ is characterised by the condition
  \begin{equation*}
    \sum_{A=I,J,K} (\xi_{\sym3 H })_A A = - \sum_{A=I,J,K} A
    (\xi_{\sym3H})_A = - \xi_{\sym3H} .  
  \end{equation*}
\item The part $\xi_H = \xi_{3H} + \xi_{KH} + \xi_{EH} $ of $\xi$ in $(
  \Lambda_0^2 E + K + E) H$ is characterised by the condition
  \begin{equation*}
    (\xi_{ H })_A A - A (\xi_H)_A - A (\xi_H) A = \xi_H,
  \end{equation*}
  for $A=I,J,K$.
\end{asparaenum}

\section{Curvature and intrinsic torsion}
\label{sec:curvature}

In order to study the contribution of the intrinsic $\SP(n)\SP(1)$-torsion
to the different components of the Riemannian curvature tensor, we consider
the $\SP(n)\SP(1)$-map $\pi_{1es} \colon \Lambda^2 T^* M \otimes \Lambda^2
T^* M \to \Lambda^2 T^* M \otimes \Lambda^2 E \sym 2 H$ defined by
\begin{equation*}
  4 \pi_{1es} (a) = 3 a - \sum_{A=I,J,K} A_{(3)}A_{(4)} a.
\end{equation*}

Let $\talt \colon T^* M \otimes T^* M \otimes \End T^* M \to \Lambda^2 T^*
M \otimes \End T^* M$ be the skewing map and define $\taltb \colon (T^* M
\otimes\End T^* M) \otimes ( T^* M \otimes\End T^* M) \to \Lambda^2 T^* M
\otimes \End T^* M$ by $\taltb (\xi \otimes \zeta)_{X,Y} Z = \xi_{\zeta_XY}
Z - \xi_{\zeta_YX} Z$ .

\begin{lemma}
  \label{projection:es}
  For the curvature tensor $R \in \FCur$, the intrinsic
  $\SP(n)\SP(1)$-torsion~$\xi$ and $\gamma_I = d \lambda_I + \lambda_J
  \wedge \lambda_K$, we have
  \begin{equation*}
    \begin{split}
      \pi_{1es} (R)&(X,Y,Z,U) \\&= \tfrac12 \sum_{A=I,J,K} \gamma_A \otimes
      \omega_A ( X,Y, Z, U) + \inp{\talt ( \Nt \xi )_{X,Y} Z}U \eqbreak -
      \tfrac34 \inp{\talt (\xi \circ \xi)_{X,Y} Z}U - \tfrac14
      \sum_{A=I,J,K} \inp{A\talt (\xi \circ \xi)_{X,Y} AZ}U \eqbreak +
      \inp{\taltb (\xi \otimes\xi)_{X,Y} Z}U.
    \end{split}
  \end{equation*}
\end{lemma}

\begin{proof}
  Since $R(X,Y,IZ,IU) - R(X,Y,Z,U) = - ( R_{X,Y} \omega_I ) (Z,IU)$, using
  the so-called Ricci formula \cite[p.~26]{Besse:Einstein}, we have
  \begin{gather} \label{ricci:form} R(X,Y,IZ,IU) - R(X,Y,Z,U) =
    \talt(\nabla^2\omega_I)_{X,Y}(Z,IU),
  \end{gather}
  where in this case $\talt \colon T^* M \otimes T^* M \otimes \Lambda^2
  T^* M \to \Lambda^2 T^* M \otimes \Lambda^2 T^* M$ is also the skewing
  map.  On the other hand, from equation~\eqref{torsion:sp(n)sp(1)}, it
  follows
  \begin{equation}
    \label{nablaomegai}
    \begin{split}
      (\nabla_X \omega_I) (Y,Z)& = \lambda_K(X) \omega_J(Y,Z) -
      \lambda_J(X) \omega_K(Y,Z) \eqbreak - \inp Y{\xi_X IZ} + \inp
      Y{I\xi_X Z}.
    \end{split}
  \end{equation}
  Now, taking $\Nt = \nabla + \xi$ into account and using repeatedly
  equation~\eqref{nablaomegai}, from the right side of
  equation~\eqref{ricci:form} we get
  \begin{equation}
    \label{dif:riir}
    \begin{split}
      (1- &I_{(3)} I_{(4)}) R (X,Y,Z,U) \\&= ( \gamma_J \otimes \omega_J +
      \gamma_K \otimes \omega_K)(X,Y,Z,U) \eqbreak + 2 \inp{\talt
      (\lambda_J \otimes K \xi I)_{X,Y} Z}U - 2 \inp{\talt (\lambda_K
      \otimes I \xi J)_{X,Y} Z}U \eqbreak + \inp{\talt (\lambda_J \otimes
      \xi J)_{X,Y} Z}U + \inp{\talt (\lambda_J \otimes J \xi )_{X,Y} Z}U
      \eqbreak - \inp{\talt (\lambda_K \otimes \xi K)_{X,Y} Z}U -
      \inp{\talt (\lambda_K \otimes K \xi )_{X,Y} Z }U \eqbreak + \inp{
      \talt ( \Nt \xi )_{X,Y} Z}U + \inp{ \talt (\Nt I \xi I)_{X,Y} Z}U
      \eqbreak- \inp{\talt (\xi \circ \xi)_{X,Y} Z}U - \inp{I\talt (\xi
      \circ \xi )_{X,Y} IZ}U \eqbreak + \inp{ \taltb (\xi \otimes
      \xi)_{X,Y} Z}U + \inp{I \taltb (\xi \otimes \xi)_{X,Y} IZ}U.
    \end{split}
  \end{equation}
  From this identity the Lemma follows.
\end{proof}

Another projection that we need to consider is $\pi_{1s} \colon \Lambda^2
T^* M \otimes \Lambda^2 T^* M \to \Lambda^2 T^* M \otimes \sym 2 H$ given
by
\begin{equation*}
  \pi_{1s} (a)(X,Y,Z,U) = \frac1{4n} \sum_{A=I,J,K} \inp{a(X,Y,
  \cdot, \cdot )}{\omega_A} \omega_A(Z,U).
\end{equation*}

\begin{lemma}
  For the curvature tensor $R \in \FCur$, the intrinsic
  $\SP(n)\SP(1)$-torsion $\xi$ and $\gamma_I = d \lambda_I + \lambda_J
  \wedge \lambda_K$, we have
  \begin{gather}
    \begin{split}
      \pi_{1s} (R)(X,Y,Z,U) &= \tfrac12 \sum_{A=I,J,K} \gamma_A \otimes
      \omega_A (X,Y,Z,U) \eqbreak + \tfrac1{4n} \sum_{A=I,J,K} \inp{\xi_X
      e_i}{\xi_Y A e_i} \omega_A(Z,U),
    \end{split}\\
    \label{astriccia}
    \Ric^*_A (X,Y) = - n \gamma_A (X,AY)  - \inp{\xi_X e_i}{\xi_{AY} Ae_i},\\
    \label{qricci:xigam}
    \Ricq (X,Y) = - n \sum_{A=I,J,K} \gamma_A (X,AY) - \sum_{A=I,J,K}
    \inp{\xi_X e_i}{\xi_{AY} Ae_i}.
  \end{gather}
\end{lemma}

\begin{proof}
  In equation~\eqref{dif:riir}, we consider $Z=Ke_i$ and $U=e_i$.
  Therefore, we obtain
  \begin{equation*}
    4 \inp{R(X,Y, \cdot, \cdot )}{\omega_K} = 2
    R(X,Y,Ke_i,e_i) = 4n  \gamma_K (X,Y) - 4 \inp{\xi_X \xi_Y  K
    e_i}{e_i}.
  \end{equation*}
  Since $2 \Ric_K^* (X,KY) = R(X,Y,Ke_i,e_i)$, the equations of the Lemma
  follow.
\end{proof}

A third projection is the map $\pi_1 \colon \Lambda^2 T^* M \otimes
\Lambda^2 T^* M \to \Lambda^2 T^* M \otimes \Lambda_0^2 E \sym 2 H$ defined
by $\pi_1 = \pi_{1es} - \pi_{1s}$.  For the curvature tensor $R \in \FCur$,
we have
\begin{equation}
  \label{pir:components}
  \begin{split}
    \pi_1 (R)(X,Y,Z,U)= & \inp{\talt ( \Nt \xi )_{X,Y} Z}U - \tfrac34
    \inp{\talt (\xi \circ \xi)_{X,Y} Z}U \smalleqbreak - \tfrac14
    \sum_{A=I,J,K} \inp{A \talt (\xi \circ \xi)_{X,Y} AZ}U + \inp{\taltb
    (\xi \otimes\xi)_{X,Y} Z}U \smalleqbreak - \tfrac1{4n} \sum_{A=I,J,K}
    \inp{\xi_X e_i}{\xi_Y A e_i} \omega_A(Z,U).
  \end{split}
\end{equation}

Let $\QK$ be the subspace of $\FCur$ such that $\QK = \FCur \cap \ker
\pi_1$.  The space $\QK$ can be seen as the space of possible curvature
tensors of a quaternionic K{\"a}hler manifold.  On $\Lambda^2 T^* M \otimes
\Lambda^2 T^* M$, we will consider the extension of the metric
$g=\inp\cdot\cdot$ defined by
\begin{equation}
  \label{scalarproduct:curvatures}
  \inp ab = a(e_{i_1}, e_{i_2}, e_{i_3}, e_{i_4}) b(e_{i_1}, e_{i_2},
  e_{i_3}, e_{i_4}), 
\end{equation}
and write $\QK^\perp$ for the orthogonal complement of $\QK$ in~$\FCur$,
i.e., $\FCur = \QK + \QK^\perp$.  There exists an $\SP(n)\SP(1)$-map
\begin{equation*}
  \pi_2 \colon \Lambda^2 T^* M \otimes \Lambda_0^2 E \sym 2 H \to \QK^\perp
\end{equation*}
such that the restriction of $\pi^\perp = \pi_2 \circ \pi_1$ to $\FCur$ is
the orthogonal projection $\FCur \to \QK^\perp$ and the restriction of
$\pi_2$ to the orthogonal complement of $\pi_1 (\FCur)$ is zero.
Therefore, making use of the $\SP(n)\SP(1)$-map $\pi_2$, we have the
following consequence of equation~\eqref{pir:components}.

\begin{proposition}
  On an almost quaternion-Hermitian manifold, the components of
  $\pi^\perp(R)$ in $\QK^\perp$ are linear functions of the components of
  $\Nt \xi$ and $\xi \otimes \xi$, where $\Nt = \nabla + \xi$ is the
  minimal $\SP(n)\SP(1)$-connection.  \qed
\end{proposition}

Since there are components of $R$ in $\QK$ and $\QK^\perp$ which only
depend on the Ricci and the q-Ricci tensors, a detailed study of these
tensors will refine the above result.

\begin{lemma}
  On an almost quaternion-Hermitian manifold, the Ricci and q-Ricci
  curvature tensors satisfy the identities
  \begin{gather}
    \label{ricmenosqric}
    \begin{split}
      3 &\Ric(X,Y) - \Ricq(X,Y) \\
      & = \sum_{A=I,J,K} (- 2 \gamma_A ( X,A Y) + \inp{\xi_X
      e_i}{\xi_{Ae_i} A Y} + \inp{\xi_X A Y}{\xi_{e_i} A e_i} ) \eqbreak +
      4 \inp{( \Nt_X \xi )_{e_i} Y}{e_i} - 4 \inp{( \Nt_{e_i} \xi )_X
      Y}{e_i} - \inp{\xi_X e_i}{\xi_{e_i} Y} \eqbreak - 3 \inp{\xi_X
      Y}{\xi_{e_i} e_i} -4 \inp{\xi_{\xi_{e_i} X} Y}{e_i},
    \end{split}\\
    \label{ricci:gamxi}
    \begin{split}
      3 \Ric & = \sum_{A=I,J,K} (-(n+2) \gamma_A (\cdot,A \cdot) -
      \inp{\xi_\cdot e_i}{\xi_{A \cdot} Ae_i} + \inp{\xi_\cdot e_i
      \xi_{Ae_i} A \cdot} ) \eqbreak + \sum_{A=I,J,K} \inp{\xi_\cdot A
      \cdot}{\xi_{e_i} Ae_i} + 4 \inp{( \Nt_\cdot \xi )_{e_i} \cdot}{e_i} -
      4 \inp{( \Nt_{e_i} \xi )_\cdot \cdot }{e_i} \eqbreak - \inp{\xi_\cdot
      e_i}{\xi_{e_i} \cdot} - 3 \inp{\xi_\cdot \cdot}{\xi_{e_i} e_i} -4
      \inp{\xi_{\xi_{e_i} \cdot} \cdot}{e_i}.
    \end{split}
  \end{gather}
\end{lemma}

\begin{proof}
  If we take $Y=U=e_i$ and write $Y$ instead of $Z$ in the equation of
  Lemma \ref{projection:es}, we will obtain equation~\eqref{ricmenosqric}.
  On the other hand, equation~\eqref{ricci:gamxi} is a direct consequence
  of equations \eqref{qricci:xigam} and~\eqref{ricmenosqric}.
\end{proof}

The next Lemma contains an algebraic result that we will need to analyse
the curvature tensor of a quaternionic K{\"a}hler manifold.

\begin{lemma}
  \label{gammas}
  Let $(\mathcal V,I,J,K,\inp\cdot\cdot)$ be a quaternionic vector space of
  dimension $4n$, $n >1$.  If $\gamma_I$, $\gamma_J$ and $\gamma_K$ are
  three two-forms such that
  \begin{equation*}
    \gamma_I \wedge \omega_J = \gamma_J \wedge \omega_I, \quad
    \gamma_J \wedge \omega_K = \gamma_K \wedge \omega_J, \quad
    \gamma_K \wedge \omega_I = \gamma_I \wedge \omega_K,
  \end{equation*}
  then $\gamma_A = c \, \omega_A$, for $A=I,J,K$, where
  \begin{equation*}
    2n \, c = \inp{\gamma_I}{\omega_I} = \inp{\gamma_J}{\omega_J} =
    \inp{\gamma_K}{\omega_K}. 
  \end{equation*}
\end{lemma}
\begin{proof}
  If we compute the contractions
  \begin{gather*}
    (\gamma_I \wedge \omega_J) (X,Y,Ie_i,e_i) = (\gamma_J \wedge
    \omega_I)(X,Y,Ie_i,e_i),\\
    (\gamma_I \wedge \omega_J) (X,Y,Je_i,e_i) = (\gamma_J \wedge
    \omega_I)(X,Y,Je_i,e_i),
  \end{gather*}
  we will get
  \begin{gather*}
    \begin{split}
      2 (n-1) \gamma_J(X,Y) &+ \inp{\gamma_J}{\omega_I} \omega_I(X,Y) \\ &=
      \gamma_I(X,KY) + \gamma_I(KX,Y) +\inp{\gamma_I}{\omega_I}
      \omega_J(X,Y),
    \end{split}
    \\
    \begin{split}
      2(n-1) \gamma_I(X,Y) &+ \inp{\gamma_I}{\omega_J} \omega_J(X,Y) \\ &=
      - \gamma_J (X,KY) - \gamma_J (KX,Y) + \inp{\gamma_J}{\omega_J}
      \omega_I(X,Y).
    \end{split}
  \end{gather*}
  As a consequence of these equations, $K\gamma_I = - \gamma_I$ and
  $K\gamma_J = - \gamma_J$, i.e., $\gamma_I$~and $\gamma_J$ are
  anti-Hermitian for~$K$.  Moreover, we get the following identities
  \begin{gather*}
    2(n-1) \gamma_I + \inp{\gamma_I}{\omega_J} \omega_J = 2 K_{(1)}
    \gamma_J + \inp{\gamma_J}{\omega_J}
    \omega_I, \\
    2 (n-1) \gamma_J + \inp{\gamma_J}{\omega_I} \omega_I = - 2 K_{(1)}
    \gamma_I + \inp{\gamma_I}{\omega_I} \omega_J .
  \end{gather*}
  By similar arguments, cyclically permuting $I$, $J$, $K$, we obtain that
  $\gamma_J$~and $\gamma_K$ are anti-Hermitian for~$I$, $\gamma_K$~and
  $\gamma_I$ are anti-Hermitian for~$J$, and
  \begin{gather}
    \label{gam:2}
    2(n-1) \gamma_I = 2 K_{(1)} \gamma_J + \inp{\gamma_J}{\omega_J}
    \omega_I =- 2 J_{(1)} \gamma_K + \inp{\gamma_K}{\omega_K} \omega_I, \\
    2 (n-1) \gamma_J = - 2 K_{(1)} \gamma_I + \inp{\gamma_I}{\omega_I}
    \omega_J = 2 I_{(1)} \gamma_K + \inp{\gamma_K}{\omega_K}\omega_J, \\
    2 (n-1) \gamma_K = - 2 I_{(1)} \gamma_J + \inp{\gamma_J}{\omega_J}
    \omega_K = 2 J_{(1)} \gamma_I+ \inp{\gamma_I}{\omega_I} \omega_K.
    \label{gam:3}
  \end{gather}
  {}From equations \eqref{gam:3}, taking equations \eqref{gam:2} into
  account, we have
  \begin{equation*}
    2 (n-1) J_{(1)} \gamma_K  =   2 K_{(1)} \gamma_J +
    \inp{\gamma_J}{\omega_J} \omega_I   = - 2  \gamma_I+
    \inp{\gamma_I}{\omega_I} \omega_I = 2(n-1) \gamma_I. 
  \end{equation*}
  Therefore
  \begin{equation*}
    \gamma_I =  J_{(1)} \gamma_K = \frac1{2n} \inp{\gamma_I}{\omega_I}
    \omega_I. 
  \end{equation*}
  Since by an analogous argument we also have $ 2n \gamma_K =
  \inp{\gamma_K}{\omega_K} \omega_K$, we find $ 2n \gamma_I = 2 J_{(1)}
  \gamma_K = \inp{\gamma_K}{\omega_K} \omega_I$.  Thus $
  \inp{\gamma_K}{\omega_K} = \inp{\gamma_I}{\omega_I}$.
\end{proof}

Now we give an alternative proof of the already classical result that any
quaternionic K{\"a}hler manifold is Einstein
\cite{Berger:CR,Ishihara:qK,Salamon:Invent}.  In our view, in the proof we
present here, the r\^ole played by the $\SP(n)\SP(1)$-structure is seen in
a more natural way.  Likewise, we also provide alternative proofs for some
known additional information about quaternionic K{\"a}hler
manifolds~\cite{Tricerri-Vanhecke:spectrum,Galicki}.

\begin{theorem}
  A quaternionic K{\"a}hler $4n$-manifold $M$, $n>1$, is Einstein,
  $q$-Einstein and locally $\Ric^*_A$-Einstein for $A=I,J,K$.  Moreover, if
  $R$ is the curvature of $M$, then
  \begin{asparaenum}
  \item\label{item:qK-Ric} $\Ric = (n+2) c\,g$, $\Ric^*_A = n c\,g$ and
    $\Ricq = 3n c\,g$, where $2 n c = \inp{\gamma_I}{\omega_I}=
    \inp{\gamma_J}{\omega_J} = \inp{\gamma_K}{\omega_K}$, and $\gamma_I = d
    \lambda_I + \lambda_J \wedge \lambda_K$;
  \item\label{item:qK-pi} $\pi_{\mathbb R_a + \mathbb R_b} (R) = \tfrac{c}8
    (\pi_2 +2 \pi_1)$, where $\pi_{\mathbb R_a + \mathbb R_b}$ is the
    projection $\FCur \to \mathbb R_a + \mathbb R_b$;
  \item\label{item:qK-R} $R \in \sym 4 E + \mathbb R (\pi_2 +2 \pi_1) =
    \QK$.
  \end{asparaenum}
\end{theorem}

\begin{proof}
  Since the manifold is quaternionic K{\"a}hler, we have
  \begin{eqnarray*}
    d \omega_I &=& \lambda_K \wedge \omega_J - \lambda_J \wedge
    \omega_K, \\
    d \omega_J &=& \lambda_I \wedge \omega_K - \lambda_K \wedge
    \omega_I, \\
    d \omega_K &=& \lambda_J \wedge \omega_I - \lambda_I \wedge
    \omega_J.
  \end{eqnarray*}
  Now, writing $\gamma_I = d \lambda_I + \lambda_J \wedge \lambda_K$, from
  $d^2 \omega_I = d^2 \omega_J = d^2 \omega_K =0$, we obtain
  \begin{equation*}
    \gamma_K \wedge \omega_J = \gamma_J \wedge  \omega_K, \quad
    \gamma_I \wedge \omega_K = \gamma_K \wedge \omega_I, \quad
    \gamma_J \wedge \omega_I = \gamma_I \wedge \omega_J.
  \end{equation*}
  Finally, using Lemma \ref{gammas} and equations \eqref{astriccia},
  \eqref{qricci:xigam} and \eqref{ricci:gamxi}, it follows that $M$ is
  Einstein and we have part~\itref{item:qK-Ric}.

  For part~\itref{item:qK-pi}, taking Propositions \ref{spn:curv1}
  and~\ref{spn:curv2} into account and using part~\itref{item:qK-Ric}, we
  have
  \begin{gather*}
    \Ric ( \pi_{\mathbb R_a} (R)) = \tfrac12 (\Ric(R) + \Ricq(R)) =
    (2n+1) c\,g, \\
    \Ric (\pi_{\mathbb R_b}(R)) = \tfrac12 (\Ric(R) - \Ricq(R)) = - (n-1) c
    \,g.
  \end{gather*}
  Now, taking Proposition~\ref{spn:curv1}\itref{item:Ra} and
  Proposition~\ref{spn:curv2}\itref{item:Rb} into account, we obtain 
  \begin{equation*}
    \pi_{\mathbb R_a} (R) = \frac{c}{12} (\pi_2 + 6 \pi_1), \quad
    \pi_{\mathbb R_b}(R) = \frac{c}{24} (\pi_2 - 6 \pi_1).
  \end{equation*}
  Hence part~\itref{item:qK-pi} follows.

  Finally, writing $R_1 = R - \pi_{\mathbb R_a + \mathbb R_b} (R)$, using
  equation~\eqref{dif:riir}, we have
  \begin{equation*}
    R_1(X,Y,Z,U) - R_1(X,Y,AZ,AU)=0,
  \end{equation*}
  for $A=I,J,K$.  From this identity it is not hard to check $L_\sigma
  (R_1) = 12 R_1$.  Then, by Proposition~\ref{glnhsplit2}, $R_1 \in \sym 4
  E$ and we have part~\itref{item:qK-R}.
\end{proof}

At this point, we can be a little more precise about the space $\QK^\perp$.

\begin{proposition}
  \label{prop:ricciqkp}
  For an almost quaternion-Hermitian manifold, if we denote $\mathbb R_\QK
  = \mathbb R (\pi_2+2\pi_1)$, then
  \begin{asparaenum}
  \item\label{item:QKpR} the orthogonal complement $\mathbb R_{\QK}^\perp$
    of $\mathbb R_\QK$ in $\mathbb R_a + \mathbb R_b$ is given by $\mathbb
    R_{\QK^\perp} = \mathbb R ( (n+2) \pi_2 - 18n \pi_1 )$;
  \item\label{item:QKp} the space $\QK^\perp$ decomposes into irreducible
    $\SP(n)\SP(1)$-modules as
    \begin{equation*}
      \begin{split}
        \QK^\perp & = \mathbb R_{\QK^\perp} + V^{22} + (\Lambda^2_0 E)_a +
        \Lambda^4_0 E + (\Lambda_0^2 E)_b + V^{31} \sym 2 H \eqbreak +
        (\sym 2 E \sym 2 H)_a + V^{211} \sym 2 H + (\sym 2 E \sym 2 H)_b +
        \Lambda^2_0 E \sym 2 H \eqbreak + V^{22} \sym 4 H + \Lambda^2_0 E
        \sym 4 H + \sym 4 H;
      \end{split}
    \end{equation*}
  \item\label{item:RQKp} the component of $R$ in $\mathbb R_\QK$ is
    determined by $\Ric ( \pi_\QK (R) )$ which is given by
    \begin{equation*}
      \Ric ( \pi_{\QK} (R) )  = \tfrac{n+2}{2(5n+1)} (  \pi_{\mathbb R}
      (\Ric) + 3 \pi_{\mathbb R} (\Ricq));
    \end{equation*}
  \item if we denote $\Ric_{\QK^\perp} = \Ric
    (\pi_{\QK^\perp}(R))$, the component of $R$ in $\mathbb
    R_{\QK^\perp}$ is determined by $\pi_{\mathbb R} ( \Ric_{\QK^\perp} )$
    which is given by
    \begin{equation*}
      \pi_{\mathbb R} ( \Ric_{\QK^\perp} ) = \tfrac{9n}{2(5n+1)}
      (\pi_{\mathbb R} (\Ric) - \tfrac{n+2}{3n} \pi_{\mathbb R}
      (\Ricq)).
    \end{equation*}
  \end{asparaenum}
\end{proposition}

\begin{proof}
  Part~\itref{item:QKpR} follows by straightforward computations.  In such
  computations we will obtain
  \begin{equation*}
    \inp{\pi_2}{\pi_2} = 36 \inp{\pi_1}{\pi_1} =
    288 n (4n-1), \qquad \inp{\pi_1}{\pi_2} = 144n.
  \end{equation*}
  We recall that the scalar product for these tensors is given by equation
  \eqref{scalarproduct:curvatures}.  Part~\itref{item:QKp} is a direct
  consequence of part~\itref{item:QKpR} and results contained in some
  previous Sections.  For parts \itref{item:QKp} and~\itref{item:RQKp},
  note that $\Ric_{\QK}$ and $\pi_{\mathbb R} ( \Ric_{\QK^\perp} )$ are the
  Ricci curvatures which respectively correspond to the components of the
  curvature in $\mathbb R (\pi_2+2\pi_1)$ and $\mathbb R ( (n+2) \pi_2 -
  18n \pi_1 )$.  Also for the q-Ricci curvatures we have
  \begin{gather*}
    \Ricq ( \pi_\QK (R) ) = \tfrac{3n}{2(5n+1)} ( \pi_{\mathbb
    R} (\Ric) + 3 \pi_{\mathbb R} (\Ricq) ), \\ 
    \pi_{\mathbb R} ( \Ricq_{\QK^\perp} ) = - \tfrac{3n}{2(5n+1)}
    (\pi_{\mathbb R} (\Ric) - \tfrac{n+2}{3n} \pi_{\mathbb R} (\Ricq)),
  \end{gather*}
  where $\Ricq_{\QK^\perp} = \Ricq (\pi_{\QK^\perp}(R))$.
\end{proof}

Now our purpose is to derive some further consequences of the identities
$d^2 \omega_I = d^2 \omega_J = d^2 \omega_K = 0$.  From equation
\eqref{nablaomegai}, we have
\begin{equation*}
  \begin{split}
    d \omega_I (X,Y,Z) & = ( \lambda_K \wedge \omega_J - \lambda_J \wedge
    \omega_K ) (X,Y,Z) \eqbreak + \sumcic_{X,Y,Z} ( \inp Y{I\xi_X Z
   } - \inp Y{\xi_X I Z} ). 
  \end{split}
\end{equation*}
Now, since $d^2 \omega_I =0$ and $\Nt = \nabla + \xi$, it is not hard to
obtain
\begin{equation}
  \label{isquare}
  \begin{split}
    0 &= (\gamma_K \wedge \omega_J - \gamma_J \wedge \omega_K ) (X,Y,Z,U)
    \eqbreak
    +
    \sumcic_{Y,Z,U} ( \inp Z{I (\Nt_X \xi)_Y U} - \inp{Z,
    (\Nt_X \xi )_Y IU} )
    \eqbreak
    - \sumcic_{X,Z,U} ( \inp Z{I (\Nt_Y \xi)_X U} - \inp{Z,
    (\Nt_Y \xi )_X IU} ) \eqbreak + \sumcic_{X,Y,U} ( \inp Y{I(\Nt_Z
    \xi)_X U} - \inp Y{(\Nt_Z \xi )_X I U} )
    \eqbreak
    - \sumcic_{X,Y,Z} ( \inp Y{I (\Nt_U \xi)_X Z} - \inp{Y,
    (\Nt_U \xi )_X IZ} ) \eqbreak - \sumcic_{Y,Z,U} ( \inp Z{(\xi_X
    I\xi)_Y U} - \inp Z{(\xi_X \xi I )_Y U} )
    \eqbreak
    + \sumcic_{X,Z,U} ( \inp Z{(\xi_Y I\xi)_X U} - \inp{Z,
    (\xi_Y \xi I )_X U} ) \eqbreak - \sumcic_{X,Y,U} ( \inp Y{(\xi_Z
    I\xi)_X U} - \inp Y{(\xi_Z \xi I )_X U} )
    \eqbreak
    + \sumcic_{X,Y,Z} ( \inp Y{(\xi_U I\xi)_X Z} - \inp{Y,
    (\xi_U \xi I )_X Z} ).
  \end{split}
\end{equation}
If we respectively replace $X,Y,Z,U$ by $X,JY,Ke_i,e_i$ in last identity,
proceed in an analogous way for $d^2 \omega_J$ and $d^2 \omega_K$, and
finally summing the obtained expressions, we get
\begin{equation}
  \label{zeroxixi}
  \begin{split}
    0 &= - \sum_{A=I,J,K} ( (2n-1) \gamma_A (X,A Y) -
    \inp{\gamma_A}{\omega_A} \inp XY) \eqbreak + \sumcic_{IJK} (
    \inp{\gamma_J}{\omega_K} \omega_I(X,Y) - \gamma_I (J X,K Y ) ) \eqbreak
    + \sum_{A=I,J,K}
    \begin{aligned}[t]
      &\bigl(- 2 \inp{\xi_X e_i}{\xi_{AY} Ae_i } - \inp{\xi_{e_i} X}{A
      \xi_{A Y } e_i} + \inp{\xi_{e_i} X}{\xi_{AY} A e_i} \eqbreak +
      \inp X{\xi_{A Y} \xi_{e_i} Ae_i} + \inp X{\xi_{\xi_{e_i} Ae_i} AY}
      + \inp{\xi_{e_i} X}{\xi_{Ae_i} A Y} \eqbreak + \inp X{\xi_{
      \xi_{e_i} A Y} A e_i} + \inp X{(\Nt_{e_i} \xi )_{Ae_i} A Y} 
      \eqbreak + \inp X{(\Nt_{e_i} \xi)_{AY} Ae_i} -\inp X{(\Nt_{A Y}
      \xi)_{e_i} A e_i} \bigr)
    \end{aligned}
    \eqbreak
    + \sumcic_{IJK}
    \begin{aligned}[t]
      &\bigl( \inp{\xi_{e_i} IX}{K\xi_{J Y} e_i} - \inp{\xi_{e_i}
      IX}{\xi_{JY} Ke_i} - \inp{IX}{\xi_{JY} \xi_{e_i} K e_i} \eqbreak -
      \inp X{I \xi_{\xi_{e_i} Ke_i} JY} +\inp{\xi_{e_i} I X, \xi_{Ke_i} J
      Y} ) + \inp X{I \xi_{ \xi_{e_i} JY} K e_i } \eqbreak - \inp X{I
      (\Nt_{J Y} \xi)_{e_i} Ke_i} +\inp X{I(\Nt_{e_i} \xi )_{J Y} K e_i}
      \eqbreak -\inp{X }{I(\Nt_{e_i} \xi)_{Ke_i} J Y} \bigr).
    \end{aligned}
  \end{split}
\end{equation}

Now by computing the $\Lambda_0^2 E$-components of the bilinear forms
contained in this identity we get
\begin{multline}
  \label{la02e} n \pi_{\Lambda_0^2 E} ( \sum_{A=I,J,K}
  \gamma_A ( \cdot,A \cdot) ) \\= - \sum_{A=I,J,K} \pi_{\Lambda_0^2 E} (
  \inp{\xi_\cdot e_i}{\xi_{A \cdot} A e_i} - \inp{\xi_{e_i}
  \cdot}{\xi_{Ae_i} A \cdot} )
  \\
  + \sum_{A=I,J,K} \pi_{\sym 2 T^*} ( \inp\cdot{\xi_{\xi_{e_i} A e_i }
  A \cdot} +\inp\cdot{(\Nt_{e_i} \xi )_{Ae_i} A \cdot}
  ).
\end{multline}

Note that if we compute the corresponding $\mathbb R$-components, we will
obtain $ \pi_{\mathbb R} ( \sum_{A=I,J,K} \gamma_A ( \cdot,A \cdot) ) =
-1/2n \sum_{A=I,J,K}\inp{\gamma_A}{\omega_A} \inp\cdot\cdot$ as it is
expected.  Finally, computing the $\sym 2 E \sym 2H $-components in
equation \eqref{zeroxixi}, we obtain
\begin{equation}
  \label{s2es2h}
  \begin{split}
    - 2 &(n-1) \pi_{\sym 2 E \sym 2 H} ( \sum_{A=I,J,K} \gamma_A ( \cdot,A
    \cdot) ) \\& = 2 \pi_{\sym 2 E \sym 2 H} ( \sum_{A=I,J,K}
    \inp{\xi_\cdot e_i}{\xi_{A \cdot} A e_i} ) 
    \eqbreak
    - \pi_{\sym 2 T^*} ( \sum_{A=I,J,K} \inp{\xi_{e_i} \cdot}{\xi_{A
    \cdot} A e_i} + \sumcic_{IJK} \inp{\xi_{e_i} I \cdot}{\xi_{ K
    \cdot } J e_i})
    \eqbreak
    + \pi_{\sym 2 T^* } ( \sum_{A=I,J,K} \inp{\xi_{A \cdot} \cdot}{
    \xi_{e_i } A e_i} + \sumcic_{IJK} \inp{\xi_{I\cdot} J \cdot}{
    \xi_{e_i} K e_i})
    \eqbreak
    + \pi_{\sym 2 T^* } ( \sum_{A=I,J,K} \inp{\xi_{e_i} \cdot}{A \xi_{A
    \cdot} e_i} + \sumcic_{IJK} \inp{\xi_{e_i} I \cdot}{J \xi_{ K
    \cdot} e_i})
    \eqbreak
    + \pi_{\sym 2 T^* } ( \sumcic_{IJK} \inp\cdot{I \xi_{\xi_{e_i}
    K\cdot } J e_i} - \sum_{A=I,J,K} \inp\cdot{\xi_{\xi_{e_i}
    A\cdot} A e_i})
    \eqbreak
    + \pi_{\sym 2 T^* } ( \sum_{A=I,J,K} \inp\cdot{(\Nt_{A \cdot}
    \xi)_{e_i} A e_i} - \sumcic_{IJK} \inp\cdot{I (\Nt_{K
    \cdot} \xi)_{e_i} Je_i} )
    \eqbreak
    - \pi_{\sym 2 T^* } ( \sum_{A=I,J,K} \inp\cdot{(\Nt_{e_i} \xi)_{A
    \cdot} A e_i} - \sumcic_{IJK} \inp\cdot{I (\Nt_{e_i}
    \xi)_{K \cdot} Je_i} ).
  \end{split}
\end{equation}

At this point we can give a more detailed description for the components of
the Ricci curvature tensors.  In fact, from equations \eqref{qricci:xigam}
and~\eqref{ricci:gamxi}, using the identity \eqref{la02e}, we get
\begin{gather}
  \allowdisplaybreaks
  \label{qricci:r} \pi_{\mathbb R} ( \Ricq ) = 
  \frac12 \sum_{A=I,J,K} ( \inp{\gamma_A}{\omega_A} -
  \frac1{2n} \inp{\xi_{e_i} e_j}{\xi_{Ae_i} A e_j} ) \inp\cdot\cdot,\\
 \label{ricci:r}
 \begin{split}
   12n \pi_{\mathbb R} ( \Ric ) & = ( - 3 \inp{\xi_{e_i} e_i}{\xi_{e_j}
   e_j} - 5 \inp{\xi_{e_i} e_j}{\xi_{e_j} e_i} + 8 \inp{( \Nt_{e_i} \xi
   )_{e_j} e_i}{e_j} \eqbreak + \sum_{A=I,J,K} ( 2(n+2)
   \inp{\gamma_A}{\omega_A} +\inp{\xi_{e_i} A e_i}{\xi_{e_j} A e_j}
   \eqbreak + \inp{\xi_{e_i} e_j}{\xi_{Ae_j}Ae_i} - \inp{\xi_{e_i}
   e_j}{\xi_{A e_i} A e_j} ) ) \inp\cdot{\cdot }.
 \end{split}
\end{gather}

Taking traces gives:

\begin{proposition}
  The scalar curvature and q-scalar curvature are 
  \begin{equation}
    \label{eq:scal}
    \begin{split}
      \scal &= \tfrac{2(n+2)}3 \sum_A \inp{\gamma_A}{\omega_A} +
      \tfrac73\norm{\xi_{33}}^2 - \tfrac13\norm{\xi_{K3}}^2 +
      \tfrac{2n^2+3n+2}{3n}\norm{\xi_{E3}}^2 \eqbreak
      - \tfrac13\norm{\xi_{3H}}^2 -
      \tfrac73\norm{\xi_{KH}}^2 + \tfrac{2(4n^2+6n+1)}{3n}\norm{\xi_{EH}}^2
      \eqbreak
      - \tfrac{16(2n+1)(n+1)}n\, d^*\theta^\xi,
    \end{split}
  \end{equation}
  and
  \begin{equation}
    \label{eq:scal-q}
    \begin{split}
      \scalq &= 2n \sum_A \inp{\gamma_A}{\omega_A} + \norm{\xi_{33}}^2 +
      \norm{\xi_{K3}}^2 + \norm{\xi_{E3}}^2 \eqbreak
      - 2 \norm{\xi_{3H}}^2 - 9
      \norm{\xi_{KH}}^2 - \tfrac23\norm{\xi_{EH}}^2.
    \end{split}
  \end{equation}
  \qed
\end{proposition}

Computing the $\Lambda^2_0E$-components gives
\begin{gather}
  \begin{split}
   \label{qricci:l02e}
   \pi_{\Lambda_0^2 E} ( \Ricq ) & = - \sum_{A=I,J,K} \pi_{\sym 2 T^*} (
   \inp\cdot{\xi_{\xi_{e_i} A e_i } A \cdot} + \inp{\cdot}{
   (\Nt_{e_i} \xi )_{Ae_i} A \cdot} ) \eqbreak - \sum_{A=I,J,K}
   \pi_{\Lambda_0^2 E} ( \inp{\xi_{e_i} \cdot}{\xi_{Ae_i} A \cdot
  } ),
 \end{split}\\
 \label{ricci:l02e}
 \begin{split}
   3 \pi_{\Lambda_0^2 E} ( \Ric ) &= - \pi_{\Lambda_0^2 E} (
   \inp{\xi_\cdot e_i}{\xi_{e_i} \cdot} + 3 \inp{\xi_\cdot \cdot}{
   \xi_{e_i} e_i} + 4 \inp{\xi_{\xi_{e_i} \cdot} \cdot}{e_i } )
   \smalleqbreak - \frac{n+2}n \sum_{A=I,J,K} \pi_{\sym 2 T^*} (
   \inp{\cdot}{\xi_{\xi_{e_i} A e_i } A \cdot} + \inp\cdot{(\Nt_{e_i} \xi
   )_{Ae_i} A \cdot} ) \smalleqbreak + \sum_{A=I,J,K} \pi_{\Lambda_0^2 E} (
   \inp{\xi_\cdot A \cdot}{\xi_{e_i} A e_i} + \inp{\xi_\cdot
   e_i}{\xi_{Ae_i} A \cdot } ) \smalleqbreak + \sum_{A=I,J,K}
   \pi_{\Lambda_0^2 E} (\frac 2{n} \inp{\xi_\cdot e_i}{\xi_{A \cdot} A
   e_i} - \frac{n+2}n \inp{\xi_{e_i} \cdot}{\xi_{Ae_i} A \cdot} )
   \smalleqbreak + 4 \pi_{\Lambda_0^2 E} ( \inp{( \Nt_\cdot \xi )_{e_i}
   \cdot}{e_i } - \inp{( \Nt_{e_i} \xi )_\cdot \cdot}{e_i} ).
 \end{split}
\end{gather}

Explicit expressions for the $\sym 2 E \sym 2H$-components $\pi_{\sym
2E\sym 2H} ( \Ricq )$ of $\Ricq$ and $ \pi_{\sym 2E\sym 2H} ( \Ric )$
of~$\Ric$ can be easily obtained from equations \eqref{qricci:xigam},
\eqref{ricci:gamxi} and~\eqref{s2es2h}.  Because of their sizes, we will
not write such expressions, but it is clear that such expressions depend
linearly on $\xi \otimes \xi$ and~$\Nt \xi$.

It remains to analyse the $\Lambda_0^2 E \sym 2 H$-component of~$\Ricq$.
For this purpose, replace $Z$ and $U$ by $Ie_i$ and~$e_i$ in
equation~\eqref{isquare}, perform analogous operations for $J$ and~$K$ and
then add the expressions obtained to get
\begin{equation*}
  \begin{split}
    0 &= 2 \sum_{A=I,J,K} (\gamma_A(X,AY) + \gamma_A (AX,Y)) \smalleqbreak
    + \sumcic_{IJK} ( \inp{\gamma_J}{\omega_K} - \inp{\gamma_K}{\omega_J})
    \omega_I (X,Y) 
    + \inp{\xi_X e_i}{\xi_{e_i} Y} - \inp{\xi_Y e_i}{\xi_{e_i} X} \smalleqbreak
    + 3 \inp{\xi_X Y}{\xi_{e_i} e_i} - 3 \inp{\xi_Y X}{\xi_{e_i} e_i}
    + 4 \inp X{\xi_{ \xi_{e_i} Y} e_i} - 4 \inp Y{ \xi_{\xi_{e_i} X} e_i}
    \smalleqbreak
    + \sum_{A=I,J,K}
    \begin{aligned}[t]
      \bigl( &- \inp{\xi_X e_i}{\xi_{Ae_i} AY} + \inp{\xi_Y e_i}{\xi_{Ae_i} A
      X} +  \inp X{\xi_{ \xi_{e_i} Ae_i} A Y} \smalleqbreak
      - \inp Y{\xi_{ \xi_{e_i} Ae_i} A X} - \inp{\xi_X AY}{\xi_{e_i} A e_i} +
      \inp{\xi_Y AX}{\xi_{e_i} Ae_i} \smalleqbreak
      + \inp{\xi_{e_i} X}{\xi_{Ae_i} A Y} - \inp{\xi_{e_i} Y}{\xi_{Ae_i} A
      X} + \inp X{(\Nt_{e_i} \xi )_{Ae_i} A Y} \smalleqbreak
      - \inp Y{(\Nt_{e_i} \xi )_{Ae_i} A X}
      \bigr)
    \end{aligned}
    \smalleqbreak
    - 4 \inp X{(\Nt_Y \xi)_{e_i} e_i} + 4 \inp Y{(\Nt_X \xi)_{e_i} e_i} + 4
    \inp X{(\Nt_{e_i} \xi)_Y e_i} - 4 \inp Y{(\Nt_{e_i} \xi )_X e_i}.
  \end{split}
\end{equation*}
{}From this last identity it is straightforward to derive the
$\Lambda_0^2 E \sym 2 H$-component of $\sum_{A=I,J,K} \gamma_A( \cdot,A
\cdot)$ which is given by
\begin{equation*}
  \begin{split}
    2 &\pi_{\Lambda_0^2 E \sym 2 H} ( \sum_{A=I,J,K} \gamma_A( X,A Y
    ) ) \\&= - \pi_{\Lambda_0^2 E \sym 2 H} ( \inp{\xi_X e_i}{\xi_{e_i}
    Y} + 3 \inp{\xi_X Y}{\xi_{e_i} e_i} )
    \eqbreak
    - 4 \pi_{\Lambda_0^2 E \sym 2 H} ( \inp X{\xi_{ \xi_{e_i} Y}
    e_i} +\inp X{(\Nt_{e_i} \xi)_Y e_i} -
    \inp X{(\Nt_Y \xi)_{e_i} e_i} )
    \eqbreak
    + \sum_{A=I,J,K} \pi_{\Lambda_0^2 E \sym 2 H} ( \inp{\xi_X e_i}{
    \xi_{Ae_i} A Y} + \inp{\xi_X A Y}{\xi_{e_i} A e_i } -
    \inp{\xi_{e_i} X}{\xi_{Ae_i} A Y} )
    \eqbreak
    - \sum_{A=I,J,K} \pi_{\Lambda^2 T^*} ( \inp X{\xi_{ \xi_{e_i} A
    e_i} A Y} + \inp X{(\Nt_{e_i} \xi )_{Ae_i} A Y } ).
  \end{split}
\end{equation*}
Therefore, using this identity and equation \eqref{qricci:xigam}, we deduce
the $\Lambda_0^2 E \sym 2 H$-component of the q-Ricci tensor which is
given by
\begin{equation}
  \label{qricci:l02es2h}
  \begin{split}
    \tfrac2{n} &\pi_{\Lambda_0^2 E \sym 2 H} ( \Ricq )(X,Y)\\ &= \pi_{\Lambda_0^2 E
    \sym 2 H} ( \inp{\xi_X e_i}{\xi_{e_i} Y} + 3
    \inp{\xi_X Y}{\xi_{e_i} e_i} )
    \smalleqbreak
    + 4 \pi_{\Lambda_0^2 E \sym 2 H} ( \inp X{\xi_{ \xi_{e_i}
    Y} e_i} +\inp X{(\Nt_{e_i} \xi)_Y e_i} -
    \inp X{(\Nt_Y \xi)_{e_i} e_i} )
    \smalleqbreak
    - \sum_{A=I,J,K} \pi_{\Lambda_0^2 E \sym 2 H} ( \inp{\xi_X
    e_i}{\xi_{Ae_i} A Y} + \inp{\xi_X A Y}{
    \xi_{e_i} A e_i} - \inp{\xi_{e_i} X}{\xi_{Ae_i} A Y
   } )
    \smalleqbreak
    + \sum_{A=I,J,K} \pi_{\Lambda^2 T^*} ( \inp X{\xi_{ \xi_{e_i} A
    e_i} A Y} + \inp X{(\Nt_{e_i} \xi )_{Ae_i} A Y
   } )
    \smalleqbreak
    - \tfrac2{n} \sum_{A=I,J,K} \pi_{\Lambda_0^2 E \sym 2 H}( \inp{\xi_X
    e_i}{\xi_{A Y} Ae_i} ). 
  \end{split}
\end{equation}

Equations \eqref{qricci:r}--\eqref{qricci:l02es2h} and the above
description of the $\sym 2 E \sym 2H$\bdash components of $\Ricq$ and
$\Ric$ give rise to the following result.  Here we will follow the notation
used in~\S\ref{sec:torsion} writing the components of the intrinsic torsion
$\xi$ as $\xi_{UF}$, for $U=3,K,E$ and $F=3,H$.

\begin{theorem}
  \label{thm:Ric}
  Let $M$ be an almost quaternion-Hermitian $4n$-manifold, $n> 1$, with
  minimal $\SP(n)\SP(1)$-connection $\Nt = \nabla + \xi$.  The tensors
  $\sum_A\inp{\gamma_A}{\omega_A}$, $\Nt\xi_{UF}$ and $\xi_{UF} \odot
  \xi_{VG}$ contribute to the components of the q-Ricci curvature $\Ricq$
  via equation~\eqref{qricci:xigam} and to the Ricci curvature $\Ric$ via
  equation~\eqref{ricci:gamxi} if and only if there is a tick in the
  corresponding place in Table~\ref{tab:Ric}.  \qed
\end{theorem}

\begin{table}[tp]
  \centering
  \begin{tabular}{lccccccc}
    \toprule
    &\multicolumn4{c}{$\Ricq$}&\multicolumn3{c}{$\Ric$}\\
    \cmidrule(lr){2-5}
    \cmidrule(lr){6-8}
    &$\mathbb R$&$\Lambda_0^2 E$&$\sym 2 E \sym 2 H$&
    $\Lambda_0^2 E \sym 2H$&$\mathbb
    R$&$\Lambda_0^2 E$&$\sym 2 E \sym 2 H$\\
    \midrule
    $\sum_A \inp{\gamma_A}{\omega_A}$&\T&  &  &  &\T&  &  \\
    \midrule
    $\Nt\xi_{33}$&  &  &\T&\T&  &  &\T\\
    $\Nt\xi_{K3}$&  &  &\T&\T&  &  &\T\\
    $\Nt\xi_{E3}$&  &  &\T&\T&  &  &\T\\
    $\Nt\xi_{3H}$&  &\T&\T&\T&  &\T&\T\\
    $\Nt\xi_{KH}$&  &\T&\T&\T&  &\T&\T\\
    $\Nt\xi_{EH}$&  &\T&\T&\T&\T&\T&\T\\
    \midrule
    $\xi_{33} \otimes\xi_{33}$&\T&\T&\T&  &\T&\T&\T\\
    $\xi_{K3} \otimes\xi_{K3}$&\T&\T&\T&  &\T&\T&\T\\
    $\xi_{E3} \otimes\xi_{E3}$&\T&\T&\T&  &\T&\T&\T\\
    $\xi_{3H} \otimes\xi_{3H}$&\T&\T&\T&  &\T&\T&\T\\
    $\xi_{KH} \otimes\xi_{KH}$&\T&\T&\T&  &\T&\T&\T\\
    $\xi_{EH} \otimes\xi_{EH}$&\T&\T&\T&  &\T&\T&\T\\
    \midrule
    $\xi_{33} \odot\xi_{K3}$&  &\T&\T&\T&  &\T&\T\\
    $\xi_{33} \odot\xi_{E3}$&  &\T&  &\T&  &\T&  \\
    $\xi_{33} \odot\xi_{3H}$&  &  &\T&\T&  &  &\T\\
    $\xi_{33} \odot\xi_{KH}$&  &  &\T&\T&  &  &\T\\
    $\xi_{33} \odot\xi_{EH}$&  &  &  &\T&  &  &  \\
    $\xi_{K3} \odot\xi_{E3}$&  &\T&\T&\T&  &\T&\T\\
    $\xi_{K3} \odot\xi_{3H}$&  &  &\T&\T&  &  &\T\\
    $\xi_{K3} \odot\xi_{KH}$&  &  &\T&\T&  &  &\T\\
    $\xi_{K3} \odot\xi_{EH}$&  &  &\T&\T&  &  &\T\\
    $\xi_{E3} \odot\xi_{3H}$&  &  &  &\T&  &  &  \\
    $\xi_{E3} \odot\xi_{KH}$&  &  &\T&\T&  &  &\T\\
    $\xi_{E3} \odot\xi_{EH}$&  &  &\T&\T&  &  &\T\\
    $\xi_{3H} \odot\xi_{KH}$&  &\T&\T&\T&  &\T&\T\\
    $\xi_{3H} \odot\xi_{EH}$&  &\T&  &\T&  &\T&  \\
    $\xi_{KH} \odot\xi_{EH}$&  &\T&\T&\T&  &\T&\T\\
    \bottomrule
  \end{tabular}
  \caption{Ricci curvatures from Theorem~\ref{thm:Ric}.}
  \label{tab:Ric}
\end{table}

Taking Proposition \ref{prop:ricciqkp} (iii) and (iv) into account, using
equations \eqref{qricci:r} and \eqref{ricci:r} we have the following
expressions which determine  the curvature components in
${\mathbb R}_\QK$ and ${\mathbb R}_{\QK^\perp}$, respectively,
\begin{gather}
  \begin{split}
    \tfrac{24n(5n+1)}{n+2} \Ric_{\QK} &= ( - 3 \inp{\xi_{e_i} e_i}{
    \xi_{e_j} e_j} - 5 \inp{\xi_{e_i} e_j}{\xi_{e_j} e_i} + 8 \inp{(
    \Nt_{e_i} \xi )_{e_j} e_i}{e_j} \eqbreak + \sum_{A=I,J,K} ( 4(5n+1)
    \inp{\gamma_A}{\omega_A} +\inp{\xi_{e_i} A e_i}{\xi_{e_j} A e_j}
    \label{rqk}
    \eqbreak + \inp{\xi_{e_i} e_j}{\xi_{Ae_j}Ae_i} - 10\inp{\xi_{e_i}
    e_j}{\xi_{A e_i} A e_j} ) ) \inp\cdot\cdot,
  \end{split}
  \\
  \label{rqkperp}
  \begin{split}
    &\tfrac{8(5n+1)}3 \pi_{\mathbb R} ( \Ric_{\QK^\perp} ) = ( - 3
    \inp{\xi_{e_i} e_i}{\xi_{e_j} e_j} - 5 \inp{\xi_{e_i} e_j}{ \xi_{e_j}
    e_i} + 8 \inp{( \Nt_{e_i} \xi )_{e_j} e_i}{e_j } \eqbreak +
    \sum_{A=I,J,K} ( \inp{\xi_{e_i} A e_i}{\xi_{e_j} A e_j} +
    \inp{\xi_{e_i} e_j}{\xi_{Ae_j}Ae_i} + 2 \inp{\xi_{e_i} e_j}{\xi_{A e_i}
    A e_j} ) )\inp\cdot\cdot.
  \end{split}
\end{gather}

Due to Proposition~\ref{spn:curv1}\itref{item:22aa} and
Proposition~\ref{spn:curv2}\itref{item:40Ebb}, the curvature components in
$( \Lambda_0^2 E)_a$ and $( \Lambda_0^2 E)_b$ are determined respectively
by $2 \Ric_{( \Lambda_0^2 E)_a} = \pi_{\Lambda_0^2 E} ( \Ric + \Ricq )$ and
$2 \Ric_{( \Lambda_0^2 E)_b} = \pi_{\Lambda_0^2 E} ( \Ric - \Ricq )$.
Using equations \eqref{qricci:l02e} and~\eqref{ricci:l02e}, we
obtain the following expressions
\begin{gather}
  \label{ricci:l02ea}
  \begin{split}
    6 \Ric_{( \Lambda_0^2 E)_a} &= - \pi_{\Lambda_0^2 E} ( \inp{\xi_\cdot e_i}{\xi_{e_i} \cdot} + 3 \inp{\xi_\cdot \cdot}{
    \xi_{e_i} e_i} + 4 \inp{\xi_{\xi_{e_i} \cdot} \cdot}{e_i
   } )
    \eqbreak
    - \tfrac{2(2n+1)}n \sum_{A=I,J,K} \pi_{\sym 2 T^*} ( \inp\cdot{
    \xi_{\xi_{e_i} A e_i } A \cdot} + \inp\cdot{(\Nt_{e_i} \xi
    )_{Ae_i} A \cdot} )
    \eqbreak
    + \sum_{A=I,J,K} \pi_{\Lambda_0^2 E} ( \inp{\xi_\cdot A \cdot}{
    \xi_{e_i} A e_i} + \inp{\xi_\cdot e_i}{\xi_{Ae_i} A \cdot
   } )
    \eqbreak
    + \sum_{A=I,J,K} \pi_{\Lambda_0^2 E} (\tfrac 2{n} \inp{\xi_\cdot
    e_i}{\xi_{A \cdot} A e_i} - \tfrac{2(2n+1)}n \inp{\xi_{e_i}
    \cdot}{\xi_{Ae_i} A \cdot} )
    \eqbreak
    + 4 \pi_{\Lambda_0^2 E} ( \inp{( \Nt_\cdot \xi )_{e_i} \cdot}{e_i
   } - \inp{( \Nt_{e_i} \xi )_\cdot \cdot }{e_i} ),
  \end{split}
  \\
  \label{ricci:l02eb}
  \begin{split}
    6 \Ric_{( \Lambda_0^2 E)_b} &= - \pi_{\Lambda_0^2 E} ( \inp{\xi_\cdot e_i}{\xi_{e_i} \cdot} + 3 \inp{\xi_\cdot \cdot}{
    \xi_{e_i} e_i} + 4 \inp{\xi_{\xi_{e_i} \cdot} \cdot}{e_i
   } )
    \eqbreak
    + \tfrac{2(n-1)}n \sum_{A=I,J,K} \pi_{\sym 2 T^*} ( \inp\cdot{
    \xi_{\xi_{e_i} A e_i } A \cdot} + \inp\cdot{(\Nt_{e_i} \xi
    )_{Ae_i} A \cdot} )
    \eqbreak
    + \sum_{A=I,J,K} \pi_{\Lambda_0^2 E} ( \inp{\xi_\cdot A \cdot}{
    \xi_{e_i} A e_i} + \inp{\xi_\cdot e_i}{\xi_{Ae_i} A \cdot
   } )
    \eqbreak
    + \sum_{A=I,J,K} \pi_{\Lambda_0^2 E} (\tfrac 2{n} \inp{\xi_\cdot
    e_i}{\xi_{A \cdot} A e_i} + \tfrac{2(n-1)}n \inp{\xi_{e_i}
    \cdot}{\xi_{Ae_i} A \cdot} )
    \eqbreak
    + 4 \pi_{\Lambda_0^2 E} ( \inp{( \Nt_\cdot \xi )_{e_i} \cdot}{e_i
   } - \inp{( \Nt_{e_i} \xi )_\cdot \cdot }{e_i} ).
  \end{split}
\end{gather}
 
Similarly, using Proposition~\ref{spn:curv3}\itref{item:3122a}
and~\itref{item:211b}, we see that the curvature components in $( \sym 2 E
\sym 2 H )_x$, for $x=a,b$, are determined by
\begin{gather*}
  \Ric_{( \sym 2 E \sym 2 H )_a} =\tfrac14 ( \Ric_{\sym 2
  E \sym 2 H} + 3 \Ricq_{\sym 2 E \sym 2 H} ),\\
  \Ric_{( \sym 2 E \sym 2 H )_b} = \tfrac34 ( \Ric_{\sym 2 E \sym 2 H} -
  \Ricq_{\sym 2 E \sym 2 H} ),
\end{gather*}
which can be given in terms of $\xi$ using equations \eqref{qricci:xigam},
\eqref{ricci:gamxi} and~\eqref{s2es2h}.

\begin{theorem}
  \label{thm:Curv}
  Let $M$ be an almost quaternion\bdash Hermitian $4n$-manifold, $4n\geqslant 8$,
  with minimal $\SP(n)\SP(1)$-connection $\Nt = \nabla + \xi $.
  \begin{asparaenum}
  \item
    \label{item:main-i}
    Using equations \eqref{qricci:xigam}, \eqref{ricci:gamxi},
    \eqref{s2es2h}, \eqref{qricci:l02es2h}, \eqref{rqk}, \eqref{rqkperp},
    \eqref{ricci:l02ea} and \eqref{ricci:l02eb} each of the tensors $
    \sum_{A=I,J,K} \inp{\gamma_A}{\omega_A}$, $\Nt\xi_{UF}$ and $\xi_{UF}
    \odot \xi_{VG}$ contributes to the components of~$R$ in $\mathbb
    R_a+\mathbb R_b$, $(\Lambda^2_0 E)_a+(\Lambda^2_0 E)_b$, $(\sym 2 E
    \sym 2 H)_a+ (\sym 2 E \sym 2 H)_b$ and $\Lambda^2_0 E \sym 2 H $ if
    and only if there is a tick in the corresponding place in
    Table~\ref{tab:CompI}.  An entry with two ticks indicates independent
    contributions to both summands.  For the modules $\mathbb R_x$,
    $\T^{p(a,b)}_{q(a,b)}$ indicates that the contribution is a positive
    multiple of $p(a,b)$ and orthogonal to $q(a,b)$ with $a=\pi_2+6 \pi_1$,
    $b=\pi_2-6 \pi_1$.
  \item
    \label{item:main-ii}
    Taking the image $\pi^\perp(R) = \pi_2 \circ \pi_1 (R)$ into account,
    where $\pi_1 (R)$ is given by equation~\eqref{pir:components}, each of
    the tensors $\Nt\xi_{UF}$ and $\xi_{UF} \odot \xi_{VG}$ contributes to
    the components of $R$ in $V^{22}$, $\Lambda^4_0 E$, $V^{31}\sym 2 H $,
    $ V^{211} \sym 2 H$, $ V^{22} \sym 4 H $, $\Lambda^2_0 E \sym 4 H$ and
    $\sym 4 H$ if and only if there is a tick in the corresponding place in
    Table~\ref{tab:Cur8}.  \qed
  \end{asparaenum}
\end{theorem}

\begin{table}[tp]
  \centering
  \begin{tabular}{llccc}
    \toprule
    $4n \geqslant 8$
    &$\mathbb R_x$&$(\Lambda^2_0 E)_x$&
    $(\sym 2 E \sym 2 H)_x$
    &$\Lambda^2_0 E \sym 2 H $ \\
    \midrule
    $\sum_A \inp{\gamma_A}{\omega_A}$
    &$\T_{k_1 a - k_2 b}^{2a+b}$&  &  &  \\
    \midrule
    $\Nt\xi_{33}$
    &  &  &\T&\T\\
    $\Nt\xi_{K3}$
    &  &  &\T&\T\\
    $\Nt\xi_{E3}$
    &  &  &\T&\T\\
    $\Nt\xi_{3H}$
    &  &\T&\T&\T\\
    $\Nt\xi_{KH}$
    &  &\T&\T&\T\\
    $\Nt\xi_{EH}$
    &$\T_{a+b}^{2k_1a-k_2b}$&\T&\T&\T\\
    \midrule
    $\xi_{33} \otimes \xi_{33}$
    &$\T_{2a+5b}^{5k_1a-k_2b}$&
    $\T$&\T& \\[1mm]
    $\xi_{K3} \otimes \xi_{K3}$
    &$\T_{2a-b}^{k_1a + k_2b}$&\T\T&\T\T&  \\[1mm]
    $\xi_{E3} \otimes \xi_{E3}$
    &$\T_{g(n)a+f(n)b}^{2k_1 f(n) a - k_2 g(n) b}$&\T&\T&  \\[1mm]
    $\xi_{3H} \otimes \xi_{3H}$
    &$\T_{5a-7b}^{-14k_1a - 5 k_2b}$&\T&\T&  \\[1mm]
    $\xi_{KH} \otimes \xi_{KH}$
    &$\T_{10a-7b}^{-17k_1a -5k_2b}$&\T\T&\T\T&  \\[1mm]
    $\xi_{EH} \otimes \xi_{EH}$
    &$\T^{-h(n)a+k_2^2 b}_{2k_1k_2a+h(n)b}$&\T&\T&  \\[1mm]
    \midrule
    $\xi_{33} \odot \xi_{K3}$
    &  &\T&\T&\T\\
    $\xi_{33} \odot \xi_{E3}$
    &  &\T&  &\T\\
    $\xi_{33} \odot \xi_{3H}$
    &  &  &\T&\T\\
    $\xi_{33} \odot \xi_{KH}$
    &  &  &\T&\T\\
    $\xi_{33} \odot \xi_{EH}$
    &  &  &  &\T\\
    $\xi_{K3} \odot \xi_{E3}$
    &  &  &\T&\T\\
    $\xi_{K3} \odot \xi_{3H}$
    &  &  &\T&\T\\
    $\xi_{K3} \odot \xi_{KH}$
    &  &  &\T\T&\T\\
    $\xi_{K3} \odot \xi_{EH}$
    &  &  &\T&\T\\
    $\xi_{E3} \odot \xi_{3H}$
    &  &  &  &\T\\
    $\xi_{E3} \odot \xi_{KH}$
    &  &  &\T&\T\\
    $\xi_{E3} \odot \xi_{EH}$
    &  &  &\T&\T\\
    $\xi_{3H} \odot \xi_{KH}$
    &  &\T&\T&\T\\
    $\xi_{3H} \odot \xi_{EH}$
    &  &\T&  &\T\\
    $\xi_{KH} \odot \xi_{EH}$
    &  &\T&  &\T\\
    \bottomrule
  \end{tabular}
  \caption{Curvature complementary (I) to $\sym 4 E$, $4n\geqslant8$, from
  Theorem~\ref{thm:Curv}\itref{item:main-i}.  Here $x=a,b$, $k_1=n-1$, $k_2
  = 2n+1$, $f(n) = n^2+3n+1$, $g(n) = n^2+1$, $h(n)=(2n-1)(n+1)$.} 
  \label{tab:CompI}
\end{table}

\begin{table}[tp]
  \centering
  \begin{tabular}{lccccccc}
    \toprule
    $4n \geqslant 8$&$V^{22}$&$\Lambda^4_0 E$
    &$V^{31}\sym 2 H $&
    $ V^{211} \sym 2 H$&$ V^{22} \sym 4 H  $
    &$\Lambda^2_0 E \sym 4 H$&$\sym 4 H$
    \\
    \midrule
    $\Nt\xi_{33}$&  &  &  &\T&  &\T&  \\
    $\Nt\xi_{K3}$&  &  &\T&\T&\T&\T&  \\
    $\Nt\xi_{E3}$&  &  &  &  &  &\T&\T\\
    $\Nt\xi_{3H}$&  &\T&  &\T&  &  &  \\
    $\Nt\xi_{KH}$&\T&  &\T&\T&  &  &  \\
    $\Nt\xi_{EH}$&  &  &  &  &  &  &  \\
    \midrule
    $\xi_{33} \otimes \xi_{33}$&\T&\T&  &\T&\T&\T&\T\\
    $\xi_{K3} \otimes \xi_{K3}$&\T&\T&\T&  &\T&\T&\T\\
    $\xi_{E3} \otimes \xi_{E3}$&  &  &  &  &  &\T&\T\\
    $\xi_{3H} \otimes \xi_{3H}$&\T&\T&  &\T&  &  &  \\
    $\xi_{KH} \otimes \xi_{KH}$&\T&\T&\T&  &  &  &  \\
    $\xi_{EH} \otimes \xi_{EH}$&  &  &  &  &  &  &  \\
    \midrule
    $\xi_{33} \odot \xi_{K3}$&\T&\T&\T&\T&\T&\T&  \\
    $\xi_{33} \odot \xi_{E3}$&  &\T&  &\T&  &\T&  \\
    $\xi_{33} \odot \xi_{3H}$&  &  &  &\T&\T&\T&\T\\
    $\xi_{33} \odot \xi_{KH}$&  &  &\T&\T&\T&\T&  \\
    $\xi_{33} \odot \xi_{EH}$&  &  &  &\T&  &\T&  \\
    $\xi_{K3} \odot \xi_{E3}$&\T&  &\T&\T&\T&\T&  \\
    $\xi_{K3} \odot \xi_{3H}$&  &  &\T&\T&\T&\T&  \\
    $\xi_{K3} \odot \xi_{KH}$&  &  &\T&\T&\T&\T&\T\\
    $\xi_{K3} \odot \xi_{EH}$&  &  &\T&\T&\T&\T&  \\
    $\xi_{E3} \odot \xi_{3H}$&  &  &  &\T&  &\T&  \\
    $\xi_{E3} \odot \xi_{KH}$&  &  &\T&\T&\T&\T&  \\
    $\xi_{E3} \odot \xi_{EH}$&  &  &  &  &  &\T&\T\\
    $\xi_{3H} \odot \xi_{KH}$&\T&\T&\T&\T&  &  &  \\
    $\xi_{3H} \odot \xi_{EH}$&  &\T&  &\T&  &  &  \\
    $\xi_{KH} \odot \xi_{EH}$&\T&  &\T&\T&  &  &  \\
    \bottomrule
  \end{tabular}
  \caption{Curvature complementary (II) to $\sym 4 E$, $4n\geqslant8$, from
  Theorem~\ref{thm:Curv}\itref{item:main-ii}.}
  \label{tab:Cur8}
\end{table}

A number of examples of almost quaternion\bdash Hermitian manifolds with
various different types of intrinsic torsion are given in Cabrera \& Swann
\cite{Cabrera-S:aqH-torsion}.

\begin{corollary}
  On an almost quaternion\bdash Hermitian manifold that is quaternionic,
  i.e., $\xi \in (\Lambda^3_0E + K + E)H$, there is no curvature in
  $V^{22}\sym 4H$, $\Lambda^2_0E\sym 4H$ or $\sym 4H$.\qed
\end{corollary}

\begin{corollary}
  If $\xi$ lies in $E(\sym3H+H)$, then there is no curvature in $V^{22}$,
  $\Lambda^4_0E$, $V^{31}\sym2H$, $V^{211}\sym2H$ or $V^{22}\sym4H$.\qed
\end{corollary}

\begin{corollary}
  For $\xi\in\Lambda^3_0E(\sym3H+H)$ there is no curvature in
  $V^{31}\sym2H$.\qed
\end{corollary}

\begin{corollary}
  Let $p\in M$.  If $\xi_p$ lies in $(\Lambda^3_0E+E)\sym
  3H+(\Lambda^3_0E+K)H$ or in $K\sym3H$ and $\pi_{\mathbb R_a+\mathbb
  R_b}(R)$ is proportional to $2\pi_1+\pi_2$ at~$p$, then $\xi_p=0$. \qed
\end{corollary}

The above result is a pointwise version of the following global theorem for
compact manifolds.  

\begin{corollary}[Bor \& Hern\'andez Lamoneda \cite{Bor-HL:aqH}]
  Suppose $M$ is compact, that $\xi\in
  (\Lambda^3_0E+E)\sym3H+(\Lambda^3_0E+K+E)H$ and that
  \begin{equation*}
    (n+2)\int_M \scalq \geqslant 3n\int_M \scal,
  \end{equation*}
  then $\xi=0$ and $M$ is quaternionic K\"ahler.
\end{corollary}

\begin{proof}
  Subtracting the right-hand side from the left, the resulting integrand is
  a sum of the form
  $r_1\norm{\xi_{33}}^2 + r_2\norm{\xi_{E3}}^2 + r_3\norm{\xi_{3H}}^2 +
  r_4\norm{\xi_{KH}}^2 + r_5\norm{\xi_{EH}}^2 + r_6d^*\theta^\xi$, with
  $r_1,\dots,r_5<0$.
\end{proof}

A similar result was found by Ivanov \& Minchev \cite{Ivanov-M:twistor} in
the special case of quaternionic K\"ahler manifolds with torsion, i.e., for
$\xi \in (K+E)H$.

\begin{corollary}
  If $\xi \in (\Lambda^3_0E+K+E)\sym3H$, then the components of the
  curvature in $(\Lambda^2_0E)_a$, $(\Lambda^2_0E)_b$, $V^{22}$ and
  $\Lambda^4_0E$ are determined by~$\xi$ tensorially.
  \qed
\end{corollary}

\begin{remark}
  It is necessary to say which formul\ae\ we use to derive the entries in
  Tables \ref{tab:Ric}, \ref{tab:CompI} and~\ref{tab:Cur8} since there are
  non-trivial relations between the tensors $\Nt\xi_{UF}$ and $\xi_{UF}
  \odot \xi_{VG}$.  These relations come from the Bianchi identity for the
  curvature~$R$ when expressed in terms of the curvature $\tilde R$ of
  $\Nt$ and $\xi$.  The modules affected are $\Lambda^2_0E\sym 2H$,
  $V^{211}\sym 2H$ and $\Lambda^2_0E\sym 4H$, there are two such relations
  for $\Lambda^2_0 E\sym 2H$ and one for each of the other two modules.
  This means that one can remove two ticks or one tick, respectively, from
  the corresponding column at the cost of introducing ticks elsewhere in
  the same column.  We expect to be able to derive these relations from the
  corresponding components of the equation $d^2\Omega=0$.
\end{remark}

\begin{small}
\providecommand{\bysame}{\leavevmode\hbox to3em{\hrulefill}\thinspace}
\providecommand{\MR}{\relax\ifhmode\unskip\space\fi MR }
\providecommand{\MRhref}[2]{%
  \href{http://www.ams.org/mathscinet-getitem?mr=#1}{#2}
}
\providecommand{\href}[2]{#2}

\end{small}

\begin{smallpars}
  Mart\'\i n Cabrera: \textit{Department of Fundamental Mathematics,
  University of La Laguna, 38200 La Laguna, Tenerife, Spain}.  E-mail:
  \url{fmartin@ull.es}

  Swann: \textit{Department of Mathematics and Computer Science, University
  of Southern Denmark, Campusvej 55, DK-5230 Odense M, Denmark}.  E-mail:
  \url{swann@imada.sdu.dk}
\end{smallpars}

\end{document}